\documentclass[a4paper,12pt,pdftex]{amsart}

\usepackage[all]{xy}
\usepackage[latin1]{inputenc}
\usepackage{amscd}
\usepackage{graphicx,subfigure}
\usepackage{rotating}
\usepackage[english]{babel}
\usepackage{amsmath,amssymb,amsfonts, amsthm,amscd, latexsym, marginnote}

\newcommand\W{\mathbf W}

\newcommand\Q{\mathbb Q}

\newcommand\G{\Gamma}

\newcommand{\A}[0]{\mathcal{A}}
\newcommand{\AT}[0]{\widetilde{\mathcal{A}}}
\newcommand{\HT}[0]{\widetilde{H}}
\newcommand{\M}[0]{\mathcal{M}}
\newcommand{\MA}[0]{\mathcal M(\A)}
\newcommand{\MT}[0]{\mathcal M(\widetilde{\A})}

\renewcommand{\L}[0]{\mathcal{L}}
\renewcommand{\l}[0]{\ell}

\newcommand{\Z}[0]{\mathbb{Z}}
\newcommand{\R}[0]{\mathbb{R}}
\newcommand{\C}[0]{\mathbb{C}}
\newcommand{\CW}[0]{\mathcal{C}}
\newcommand{\Dp}[0]{\mathcal{D}}

\newcommand{\ph}[0]{\varphi}

\newcommand{\too}[0]{\longrightarrow}

\newcommand{\Zt}[0]{\Z[t^{\pm 1}]}
\newcommand{\kf}[0]{\text{ker}\varphi}
\newcommand{\oA}{{\overline{\A}}}
\newcommand{\bs}{\bigskip}
\newcommand{\bF}[0]{\mathbf{F}}
\newcommand{\bG}[0]{\mathbf{G}}
\newcommand{\bN}[0]{\mathbf{N}}
\newcommand{\bK}[0]{\mathbf{K}}
\newcommand{\Os}[0]{\overline{s}}

\def\qed{\ifmmode $\Box$ \else{\unskip\nobreak\hfil
\penalty50\hskip1em\null\nobreak\hfil $\Box$
\parfillskip=0pt\finalhyphendemerits=0\endgraf}\fi}

\newcommand{\eq}[1][r]
       {\ar@<-3pt>@{->}[#1]
        \ar@<-1pt>@{}[#1]|<{}="gauche"
        \ar@<+0pt>@{}[#1]|-{}="milieu"
        \ar@<+1pt>@{}[#1]|>{}="droite"
        \ar@/^2pt/@{-}"gauche";"milieu"
        \ar@/_2pt/@{-}"milieu";"droite"}
\newcommand{\imm}[1][r] {\ar@{^{(}->}[#1]}

\renewcommand{\ni}[0]{\noindent}

\newcommand{\elle}[0]{ \sl{l}}

\newtheorem{df}{Definition}[section]

\newtheorem{teo}{Theorem}
\newtheorem{prop}[df]{Proposition}
\newtheorem{lem}[df]{Lemma}
\newtheorem{cor}[df]{Corollary}
\newtheorem{rmk}[df]{Remark}

\begin{document}

\hyphenation{mul-ti-pli-ci-ty ho-mo-lo-gy}

\title[Twisted cohomology of arrangements]{Twisted cohomology of arrangements of lines and Milnor fibers}

\author[M.~Salvetti]{M.~Salvetti}
\address{Department of Mathematics,
University of Pisa, Pisa Italy}
\email{salvetti@dm.unipi.it}

\author[M.~Serventi]{M.~Serventi}
\address{Department of Mathematics,
University of Pisa, Pisa Italy}
\email{serventi@mail.dm.unipi.it}

\subjclass[2000]{55N25; 57M05}

\date{\today}

\maketitle

\begin{abstract}Let $\A$ be an arrangement of affine lines in $\C^2,$ with complement $\M(\A).$ The (co)homo-logy of $\M(\A)$ with twisted coefficients is  strictly related to the  cohomology of the  Milnor fibre associated to the conified arrangement, endowed with the geometric monodromy. Although several partial results are known, even the first Betti number of the Milnor fiber is not understood.  We give here a vanishing conjecture for the first homology, which is of  a different nature with respect to the known results. Let $\Gamma$ be the graph of \emph{double points} of $\A:$ 
we conjecture that if $\Gamma$  is connected  then the geometric monodromy acts trivially on the first homology of the Milnor fiber (so the first Betti number is combinatorially determined in this case).  This conjecture depends only on the combinatorics of $\A.$  We prove it in some cases with stronger hypotheses. 

In the final parts, we introduce a new description in terms of the group given by the quotient ot the commutator subgroup of $\pi_1(\M(\A))$ by the commutator of its \emph{length zero subgroup.} We use that to deduce some new interesting cases of a-monodromicity, including a proof of the conjecture under some extra conditions.
\end{abstract}

\section{Introduction} Let $\A:=\{\l_1,\dots,\l_n\}$ be an arrangement of affine  lines in $\C^2,$ with complement $\M(\A).$  Let $\L$ be a rank$-1$ local system on $\M(\A),$ which is defined by a unitary commutative ring $R$ and an assignment of an invertible element $t_{i}\in R^*$  for each line $\l_i\in\A.$ Equivalently, $\L$ is defined by a module structure on $R$ over the fundamental group of $\M(\A)$ (such structure factorizes through the first homology of $\M(\A)$).    By  "coning" $\A$ one obtains a three-dimensional central arrangement, with complement  fibering over $\C^*$. The Milnor fiber $F$ of such fibration is a surface of degree $n+1,$ endowed with a natural monodromy automorphism of order $n+1.$ It is well known that the trivial (co)homology 
of $F$ with coefficients in a commutative ring $A,$ as a module over the monodromy action, is obtained by the (co)homology of $\M(\A)$ with coefficients in $R:=A[t^{\pm1}],$ where here the structure of $R$ as a $\pi_1(\M(\A))$-module is given by taking all the $t_i$'s equal to $t$  and the monodromy action corresponds to $t-$multiplication.  For reflection arrangements, relative to a Coxeter group $\W,$  many computations were done, especially for the {\it orbit space}    \  $\M_{\W}(\A):=\M(\A)/\W,$ which has an associated Milnor fiber $F_{\W}:=F/\W:$ in this case we know  a complete answer for $R=\Q[t^{\pm 1}],$ for all groups of finite type (see \cite{frenkel, dec_proc_sal, dec_proc_sal_stu}), and for some groups of affine type (\cite{calmorsal1, calmorsal2, calmorsal3}) (based on the techniques developed in  \cite{salvetti2, deconcini_salvetti}).    For $R=\Z[t^{\pm 1}]$ a complete answer is known in case $A_n$ (see \cite{callegaro:classical}). Some results are known for (non quotiented) reflection arrangements (see \cite{settepan}, \cite{papamaci}). A big amount of work in this case has been done on related questions, when $R=\C,$ in that case the $t_{i}$'s being non-zero complex numbers, trying to understand the jump-loci (in $(\C^*)^n$) of the cohomology (see for example \cite{suciu, cohenorlik, dimpapsuc, libyuz, falk, cohensuciu}).

 Some algebraic complexes computing the twisted cohomology of $\M(\A)$ are known (see for example the above cited papers).  In \cite{gaiffi_salvetti}, the minimal cell structure of the complement which was constructed in \cite{SS}  (see \cite{dimcapapa, randell}) was used to find an algebraic complex which computes the twisted cohomology, in the case of real defined arrangements (see also \cite{gaimorsal}). The form of the boundary maps depends not only on the lattice of the intersections which is associated to $\A$ but also on its {\it oriented matroid:} for each singular point $P$ of multiplicity $m$ there are $m-1$ generators in dimension $2$ whose boundary has non vanishing components along the lines contained in the "cone" of $P$ and passing above $P.$  
 
Many of the specific examples of arrangements with non-trivial cohomology (i.e., having non-trivial monodromy) which are known are based on the theory of {\it nets} and {\it multinets} (see \cite {falkyuz}): there are relatively few arrangements with non trivial monodromy in cohomology and some conjecture claim very strict restrictions for line arrangements (see \cite{YoBand2}).   
 
In this paper we state a vanishing conjecture of a very different nature, which is very easily stated  and which involves only  the lattice associated to the arrangement. Let $\G$  be the graph with vertex set $\A$ and edge set which is given by taking an edge $(\l_i,\l_j)$ iff $\l_i\cap \l_j$ is a double point. Then our conjecture is as follows:
\bigskip
 
 \begin{equation} \label{conj}{\underline {\rm Conjecture}:} {\mbox{\emph{ \ Assume that $\G$ is connected; then $\A$ has trivial monodromy. }}}
 \end{equation}
 \bigskip
 
\noindent This conjecture is supported by several "experiments", since all computations we made confirm it. Also, all non-trivial monodromy examples which we know have disconnected graph $\G.$ 
 We give here a proof holding with further restrictions. Our method uses the algebraic complex given in \cite{gaiffi_salvetti} so our arrangements are real. 

An arrangement with trivial monodromy will be called \emph{a-monodromic.} We also introduce a notion of monodromic triviality over $\Z.$  By using free differential calculus,  we show that $\A$ is a-monodromic over $\Z$ iff   the fundamental group of the complement $\M(A)$ of the arrangement is commutative modulo the commutator subgroup of the length-zero subgroup of the free group $F_n.$  As a consequence, we deduce that if $G:=\pi_1(\M(\A))$ modulo its second derived group is commutative, then $\A$ has trivial monodromy over $\Z.$ 

In the final part we give an intrinsic characterization of the a-monodromicity. Let $K$ be the kernel of the \emph{length map}  $G\to\Z.$ We introduce the group  $H:=\frac{[G,G]}{[K,K]},$  and we show that such group exactly measures the ''non-triviality'' of the first homology of the Milnor fiber $F,$ as well as its torsion. Any question about the first homology of $F$ is actually a question about $H.$ To our knowledge,  $H$ appears here for the first time (a preliminary partial version is appearing in \cite{SSv}).       
We use this description to give some interesting new results about the a-monodromicity of the arrangement.  First, we show that if $G$ decomposes as a direct product of two groups, each of them containing an element of length $1,$ then $\A$ is a-monodromic (thm.\ref{teo:direct}). This includes the case when $G$ decomposes as a direct product of free groups. As a further interesting  consequence, an arrangement which decomposes into two subarrangements which intersect each other transversally, is a-monodromic.

Also, we use this description to prove our conjecture under the hypotheses that we have a connected \emph{admissible} graph of commutators (thm.\ref{teo:admissible}): essentially, this means to have enough double points $\l_i\cap\l_j$  which give as relation (mod $[K,K]$) the  commutator of the fixed geometric generators $\beta_i, \beta_j$ of $G.$ 

After having finished our paper, we learned about the paper \cite{bailet} were the graph of double points is introduced and some partial results are shown, by very different methods.

\section{Some recalls.}\label{sec:char-res-var}

We recall here some general constructions (see \cite{suciu3}, also as a reference to  most of the recent literature).
Let $M$ be a space with the homotopy type of a finite CW-complex with $H_1(M;\Z)$ free abelian of rank $n,$ having basis $e_1,\dots,e_n.$ Let $\underline{t}=(t_1,\dots,t_n)\in 
(\C^*)^n$ 
and denote by $\C_{\underline{t}}$ the abelian rank one local system over $M$ given by the representation
$$\phi:H_1(M;\Z)\too \C^*=Aut(\C)$$
assigning $t_i$ to $e_i$.
\begin{df}
 With these notations one calls
 $$V(M)=\{\underline{t}\in(\C^*)^n: \dim_\C H_1(M;\C_{\underline{t}})\geq 1 \}$$
 the (first) characteristic variety of $M$
\end{df}
There are several other analogue definitions in all (co)homological dimensions, as well as refined definitions keeping into account the  dimension actually reached by the local homology groups. For our purposes here we need to consider only the above definition. 

The characteristic variety of a CW-complex $M$ turns out to be an algebraic subvariety of the algebraic torus 
$(\C^*)^{b_1(M)}$ which depends only on the fundamental group $\pi_1(M)$ (see for ex. \cite{cohensuciu}).

Let now $\A$ be a complex hyperplane arrangement in $\C^n$. One knows that the complement $\M(A)=\C^n\setminus 
\bigcup_{H\in\A}H$ has the homotopy type of a finite CW-complex of dimension $n.$ Moreover, in this case one knows by a general result  (see \cite{arapura}) that  the characteristic variety of $M$ is a finite union of torsion translated subtori of the 
algebraic torus $(\C^*)^{b_1(M)}.$

Now we need to briefly recall two standard constructions in arrangement theory (see \cite{orlik_terao} for details).

Let $\A=\{H_1,\dots,H_n\}$ be an affine hyperplane arrangement in $\C^n$ with coordinates $z_1,\dots,z_n$ and, for 
every $1\leq i\leq n$ let $\alpha_i$ be a linear polynomial such that $H_i=\alpha_i^{-1}(0)$. The \emph{cone} $c\A$ of 
$\A$ is a central arrangement in $\C^{n+1}$ with coordinates $z_0,\dots,z_n$ given by 
$\{\widetilde{H_0},\widetilde{H_1},\dots,\widetilde{H_n}\}$ where $\widetilde{H_0}$ is the coordinate hyperplane 
$z_0=0$ and, for every $1\leq i\leq n$, $\widetilde{H}_i$ is the zero locus of the homogenization of 
$\alpha_i$ with respect to $z_0$.

Now let $\AT=\{\widetilde{H}_0,\dots,\widetilde{H}_n\}$ be a central arrangement in 
$\C^{n+1}$ and choose coordinates $z_0,\dots,z_n$ such that $\widetilde{H}_0=\{z_0=0\}$; moreover, for every $1\leq 
i\leq n$, let $\widetilde{\alpha_i}(z_0,\dots,z_n)$ be such that $\widetilde{H_i}=\widetilde{\alpha_i}^{-1}(0)$. The 
\emph{deconing} of $\widetilde{\A}$ is the arrangement $d\AT$ in $\C^n$ given by $\{H_1,\dots,H_n\}$ where, if we set 
for every $1\leq i\leq n$, $\alpha_i(z_1,\dots,z_1)=\widetilde{\alpha_i}(1,z_1,\dots,z_n)$, $H_i=\alpha_i^{-1}(0)$.
One see easily that  $\M(c\A)=\M(\A)\times \C^*$  (and conversely 
$\M(\AT)=\M(d\AT)\times \C^*$).

%
The fundamental group $\pi_1(\MT))$  is generated by elementary loops $\beta_i, $ $i=0,\dots,n,$  around the hyperplanes and
in the decomposition  $\pi_1(\M(\A))\simeq \pi_1(\M(d\A))\times \Z$ the generator of 
$\Z=\pi_1(\C^*)$ corresponds to a loop going around all the hyperplanes. The generators can be ordered so that such a loop is represented by  $\beta_0\dots\beta_n.$ Choosing $\HT_0$ as the hyperplane at infinity in the deconing $\A=d\AT,$  one has
(see \cite{cohensuciu})
$$V(\AT)=\{\underline{t}\in(\C^*)^{n+1}: (t_1,\dots,t_{n})\in V(d\A) \ \text{ and } t_0\cdots t_n=1 \}.$$
It is still an open question whether  the characteristic 
variety $V(\AT)$  is combinatorially determined, that is, determined by the intersection lattice $L(\AT)$. Actually, the 
question is partially solved: thanks to the above description  we can write
$$V(\AT)=\check{V}(\AT)\cup T(\AT)$$
where $\check{V}(\AT)$ is the union of all the components of $V(\AT)$ passing through the unit element $\underline{1}=(1,1,\dots,1)$ and 
$T(\AT)$ is the union of the translated tori of $V(\AT)$.

The "homogeneous" part $\check{V}(\AT)$ is combinatorially described through the \emph{resonance variety}
$$\mathcal{R}^1(\AT):=\{a\in A^1 : H^1(A^\bullet,a\wedge \cdot)\neq 0 \}$$ 
introduced in \cite{falk}. Here $A^\bullet$ is the Orlik-Solomon algebra over $\C$ of $\AT$.
Denote by $\mathcal{V}(\AT)$ the tangent cone of $V(\AT)$ 
at $\underline{1}$; it turns out that $\mathcal{V}(\AT)\cong\mathcal{R}^1(\AT)$.  So, 
from $\mathcal{R}^1(\AT)$ we can obtain the components of $V(\AT)$ containing $\underline{1}$ by exponentiation.

It is also known (see \cite{cohensuciu,libyuz}) that $\mathcal{R}^1(\AT)$ is  a subspace arrangement: 
$\mathcal{R}^1(\AT)=C_1\cup\dots\cup C_r$ with $\dim C_i\geq 2$, $C_i\cap C_j=\emptyset$ for every $i\neq j$.

One makes a distinction between 
 \emph{local components} $C_I$ of $R^1(\AT),$ associated to a codimensional-2 flat $I$ in the intersection lattice,  which are contained in some coordinate hyperplanes; and 
 \emph{global components}, which are not contained in any 
coordinate hyperplane of $A^1$. Global components of dimension $k-1$ are known to correspond to $(k,d)$-\emph{multinets} (\cite{falkyuz}). Let $\overline{\A}$ be the projectivization of $\AT.$ A \emph{$(k,d)$-multinet} on a multi-arrangement $(\overline{\A},m),$    is a pair 
$(\mathcal{N},\mathcal{X})$ where $\mathcal{N}$ is a partition of $\overline{A}$ into $k\geq 3$ classes 
$\overline{\A}_1,\dots,\overline{\A}_k$ and 
$\mathcal{X}$ is a set of multiple points with multiplicity greater than or equal to $3$ 
which satisfies a list of conditions. We just recall that 
  $\mathcal{X}$ determines $\mathcal{N}$: construct a graph 
$\Gamma'=\Gamma'(\mathcal{X})$ with $\oA$ as vertex set and an edge from $l$ to $l'$ iff $l\cap l'\notin \mathcal{X}$. 
Then the connected components of $\Gamma'$ are the blocks of the partition $\mathcal{N}.$

\section{The Milnor fibre and a conjecture}
Let \ $Q:\C^3\to \C$ \ be a homogeneous polynomial (of degree $n+1$) which defines the arrangement $\AT.$
Then  $Q$ gives a fibration 
\begin{equation} \label{eq:fibering}
Q_{|\M(\tilde{\A})}:\ \M(\tilde{\A})\to \C^*
\end{equation}
with \emph{Milnor fibre} 
$$\mathbf F = Q^{-1}(1)$$
and \emph{geometric monodromy}   
$$\pi_1(\C^*,1)\to Aut(F)$$
induced by   \ $x\to e^{\frac{2\pi i}{n+1}}\cdot x$ \ (see for example  \cite{suciu2}, \cite{YoBand3}).

Let \  $A$ \ be any unitary commutative ring and 
$$R\ :=\ A[t,t^{-1}].$$  
Consider the abelian representation 
$$\pi_1(\MT)\to H_1(\MT;\Z)\to  Aut(R):\ \beta_j \to q\cdot$$taking a  generator $\beta_j$ into $t$-multiplication. 
Let  $R_t$  be the ring $R$ endowed with  this  \ $\pi_1(\MT)-$module structure. 
Then it is well-known:

\begin{prop} One has an $R$-module isomorphism  
$$H_*(\MT,R_t) \cong H_*(F,A)$$
where $t-$multiplication on the left corresponds to the monodromy action on the right.
\end{prop}

In particular for  $R=\Q[t,t^{-1}],$ which is a PID, one has 
$$H_*(\MT,\Q[t^{\pm 1}]) \cong H_*(F,\Q).$$
Since the monodromy operator has order dividing $n+1$ then 
$H_*(\MT;R_t)$ decomposes into cyclic modules either isomorphic to $R$ or to $\frac{R}{(\ph_d)},$ where
$\ph_d$ is a cyclotomic polynomial, with $d|n+1\ .$ 
It is another open problem to  find a (possibly combinatorial) formula for the Betti numbers of $F.$  

It derives from the spectral sequence associated to (\ref{eq:fibering}) that
$$n+1=dim(H_1(\MT;\Q))= 1 + dim \frac{H_1(F;\Q)}{(\mu-1)}$$
where on the right one has the coinvariants w.r.t. the monodromy action.
Therefore
$$b_1(F)\geq\ n;$$
actually
$$b_1(F)=n\quad   \Leftrightarrow \quad \mu=id.$$

\begin{df}\label{amonodromy} An arrangement $\AT$ with trivial monodromy will be called \emph{a-monodromic.}
\end{df}

\begin{rmk} The arrangement $\AT$ is a-monodromic iff 
$$H_1(F;\Q)\cong \Q^n \ \mbox{(equivalently: $H_1(\MT;R)\cong \left(\frac{R}{(t-1)}\right)^{n}$ )} $$
\end{rmk} 

Let $\A=d\AT$ be the affine part.  In analogy with definition \ref{amonodromy} we say

\begin{df}  The affine arrangement $\A$ is \emph{a-monodromic} if  \ 
$$H_1(\MA;R)\cong \left(\frac{R}{(t-1)}\right)^{n-1}.$$
\end{df}

By Kunneth formula one easily gets  (with $R=\Z[t^{\pm 1}]$ or $R=\Q[t^{\pm 1}]$) 
\begin{equation}\label{kunneth}H_1(\MT;R)\cong H_1(\MA;R)\otimes \frac{R}{(t^{n+1}-1)}\oplus \frac{R}{(t-1)}.\end{equation}

It follows that  if $\A$ has trivial monodromy  then $\AT$ does.  The converse is not true in general (see the example in fig.(\ref{completetriangle})).

We can now state the conjecture presented in the introduction.
\bs

{\bf\emph{Conjecture 1:}}\label{conj1}  let $\Gamma$ be the graph with vertex set $\A$ and edge-set all pairs $(\l_i,\l_j)$\label{conj2} such that $\l_i\cap\l_j$ is a double point. Then if $\Gamma$ is connected then $\A$ is a-monodromic.

{\bf\emph{Conjecture 2:}}  let $\Gamma$ be as before. Then if $\Gamma$ is connected then  $\AT$ is a-monodromic.
\bs

By formula (\ref{kunneth}) conjecture $1$ implies conjecture $2.$ 

A partial evidence of these conjecture is that the connectivity condition on the graph  of double points give strong restrictions on the characteristic variety, as we now show.
\begin{rmk}\label{noglobal} Let $\underline{t}=(t,\dots,t)\in(\C^*)^{n+1}$ give non-trivial monodromy for the arrangement $\AT.$ Then $\underline{t}\in V(\AT).$ Moreover,  $\underline{t}$ can intersect $\check{V}(\AT)$  only in some global component.
\end{rmk} 
Next theorem shows how the connectivity of $\Gamma$ is an obstruction to the existence of multinet structures.
\begin{teo}\label{teo:no-multinet}
If the above graph $\Gamma$ is connected  then the projectivized $\oA$ of $\AT$ does not support any multinet structure.
\end{teo}
\proof
Choose a set $\mathcal{X}$ of 
points of multiplicity greater than or equal to $3$ and build $\Gamma'(\mathcal{X})$ as 
we said at the end of section \ref{sec:char-res-var}. This graph $\Gamma'(\mathcal{X})$ has $\oA$ as set of vertices and  the set of edges of $\Gamma$ is contained in the set of edges of $\Gamma'(\mathcal{X})$. Since by hypothesis
$\Gamma$ is connected  then $\Gamma'(\mathcal{X})$ has at most two connected components and so $\mathcal{X}$ cannot give a 
multinet structure an $\oA$.
\qed
\begin{cor}
 If the graph $\Gamma$ is connected,  there is no global resonance component in $\mathcal{R}^1(\AT)$.
\end{cor} \qed

So, according to remark \ref{noglobal}, if $\Gamma$ is connected then non trivial monodromy could appear only in the presence of some translated subtori in the characteristic variety.

 

\section{Algebraic complexes}\label{part2} We shall prove the conjectures with extra assumptions on the arrangement.  Our tool will be  an algebraic complex which was  obtained in \cite{gaiffi_salvetti}, as a $2-$dimensional refinement of that in \cite{SS}, where the authors used the explicit construction of a {\it minimal cell complex} which models the complement. Since these complexes work for real defined arrangements, this will be our first restriction.

Of course, there are other algebraic complexes computing local system cohomology (see the references listed in the introduction). The one in \cite{gaiffi_salvetti} seemed to us particularly suitable to attack the present problem (even if we were not able to solve it in general).

First, the complex depends on a fixed and generic system of "polar coordinates". In the present situation, this just means to take an oriented affine real line $\l$  which is transverse to the arrangement. We also assume (even if it is not strictly necessary) that $\l$ is "far away" from $\A,$ meaning that it does not intersect the closure of the bounded facets of the arrangement. This is clearly possible because the union of bounded chambers is a compact set (the arrangement is finite).    The choice of $\l$ induces a labelling on the lines $\{ \l_1,\dots,\l_n\}$ in $\A,$ where the indices of the lines agree with the  ordering  of the  intersection points with $\l,$ induced by the orientation of $\l.$        

Let  us choose a basepoint $O\in\l,$  coming before all the intersection points of $\l$ with $\A$ (with respect to the orientation of $\l$).   We recall the construction in \cite{gaiffi_salvetti}  in the case  of the abelian local system defined before.

Let $Sing(\A)$ be the set of singular points of the arrangement. For any point $P\in Sing(\A)$, let $S(P):=\{\l\in \A: \ P\in \l\};$   so  $m(P)=|S(P)|$ is the multiplicity of $P.$

Let   $i_P,\ i^P$ be the minimum and maximum index of the lines in $S(P)$ (so $i_P< i^P$). We denote by $C(P)$  the subset of lines in $\A$ whose indices belong to the closed interval $[i_P,i^P].$ We also denote by 
$$U(P):=\{\l \in \A:\   \mbox {$\l$ does not separates $P$ from the basepoint  $O$} \}$$ 
 
Let  $(\mathcal C_*,\partial_*)$ be the $2-$dimensional algebraic complex of free $R-$modules having one  $0-$dimensional basis element $e^0,$ \ $n$ $1-$dimensional basis elements $e_j^1,\ j=1,\dots,n,$ ($e^1_j$ corresponding to the line $\l_j$) and {\small $\nu_2=\sum_{P\in Sing(\A)}  m(P)\!\!-\!\!1$} \  $2-$dimensional basis elements: to the singular point $P$ of multiplicity $m(P)$ we associate generators $e^2_{P,h},\ h=1,\dots, m(P)-1$ . The lines through $P$ will be indicized as $\l_{j_{P,1}},\dots, \l_{j_{P,m(P)}}$ (with growing indices). 

As a dual statement to  \cite{gaiffi_salvetti}, thm.2,  we obtain:

\begin{teo}
\label{teo:delta2}
The local system homology  $H_*(\M(\A);R)$  is computed by the complex $(\mathcal C_*,\partial_*)$ above, where 
$$\partial_1(e^1_j)\ =\ (t_j-1)\ e^0\  $$
and
{\small \begin{equation}\begin{aligned}\ensuremath{&\label{delta2gen}\partial_2(e^2_{P,h})\ =\ 
\sum_{ {\tiny
\begin{array}{c}
   \l_j \in S(P)  
\end{array}} } \left(\prod_{{\tiny
\begin{array}{c} i <j \ s.t. \\  \elle_i\in  U(P)\end{array}} }t_i    \right) \left(  \prod_{{\tiny
\begin{array}{c}   i\in [j_{P,h+1}\rightarrow j)\end{array}} } t_i-\prod_{{\tiny
\begin{array}{c} i <j \ s.t. \\  \elle_i\in  S(P)\end{array}} }t_i     \right )\ e^1_{j}\  + \\ }
\ensuremath{&
+\sum_{  {\tiny
\begin{array}{c}
\l_j\in  C(P)\cap U(P)
\end{array}}
} \left(\prod_{{\tiny
\begin{array}{c} i <j\ s.t. \\  \elle_i\in  U(P)\end{array}} } t_i    \right)    \left ( 1- \prod_{{\tiny
\begin{array}{c} i\leq j_{P,h},\ i<j  \\ \elle_i\in S(P) \end{array}} } t_i\right )\left (\prod_{{\tiny
\begin{array}{c} i\geq j_{P,h+1}, \ i <j  \\ \elle_i\in S(P) \end{array}} } t_i- \prod_{{\tiny
\begin{array}{c} i\geq j_{P,h+1}  \\  \elle_i\in S(P)\end{array}} } t_i \right ) \ e^1_{j}  }
\end{aligned}
\end{equation} }

\bigskip

\noindent where  \([j_{P,h+1}\rightarrow j) \) is the set of indices of the lines in $S(P)$ which run from $j_{P,h+1}$ (included) to $j$ (excluded) in the cyclic ordering of $1,\dots,n.$ 
\end{teo}
By convention, a product over an empty set of indices equals $1.$

When $R=A[t^{\pm 1}]$ and $t_i=t,\ i=1,\dots, n,$ we obtain the local homology  $H_*(\M(\A);R)$ by using an analog algebraic complex, where all $t_i$'s equal $t$ in the formulas. In particular (\ref{delta2gen}) becomes

{\small \begin{equation}\begin{aligned}\ensuremath{&\label{delta2t}\partial_2(e^2_{P,h})\ =\ 
\sum_{ {\tiny
\begin{array}{c}
   \l_j \in S(P)  
\end{array}} }   t^{\#\{\l_i\in U(P):\ i<j\} } \left(  t^{\#[j_{P,h+1}\rightarrow j)}
 - t^{\#\{\l_i\in S(P):\ i<j\}}   \right )\ e^1_{j}\  + \\ }
 \ensuremath{& +\sum_{  {\tiny
\begin{array}{c}
\l_j\in  C(P)\cap U(P)
\end{array}}
}  t^{\#\{\l_i\in U(P):\ i<j\} + \#\{\l_i\in S(P):\ i\geq j_{P,h+1},\ i<j\} } 
\left ( 1- t^{\#\{\l_i\in S(P):\ i\leq j_{P,h},\ i<j\} }\right) \cdot \\}
\ensuremath{& \cdot\left( 1 - \ t^{\#\{\l_i\in S(P):\ i\geq j_{P,h+1},\  i\geq j\} }
\right ) \ e^1_{j} }
\end{aligned}\end{equation}}
\bigskip

By separating in the first sum the case $j\geq j_{P,h+1}$ from the case $j\leq j_{P,h}$ we have:
\bigskip
 
{\small \begin{equation}\begin{aligned}\ensuremath{&\label{delta2tsep}\partial_2(e^2_{P,h})\ =\ 
\sum_{ {\tiny
\begin{array}{c}
   \l_j \in S(P) \\ j \geq j_{P,h+1}  
\end{array}} }   t^{\#\{\l_i\in U(P):\ i<j\} + \#\{\l_i\in S(P):\ j_{P,h+1}\leq i< j\}} \left(  
1 - t^{\#\{\l_i\in S(P):\ i\leq j_{P,h}\}}   \right )\ e^1_{j}\  +\\ }
\ensuremath{& + \sum_{ {\tiny
\begin{array}{c}
   \l_j \in S(P) \\ j \leq j_{P,h}  
\end{array}} }   t^{\#\{\l_i\in U(P):\ i<j\} + \#\{\l_i\in S(P):\  i< j\}} \left(  
t^{\#\{\l_i\in S(P):\ j_{P,h+1}\leq i \}} - 1   \right )\ e^1_{j}\  + \\ }
\ensuremath{& +\sum_{  {\tiny
\begin{array}{c}
\l_j\in  C(P)\cap U(P)
\end{array}}
}  t^{\#\{\l_i\in U(P):\ i<j\} + \#\{\l_i\in S(P):\ i\geq j_{P,h+1},\ i<j\} } 
\left ( 1- t^{\#\{\l_i\in S(P):\ i\leq j_{P,h},\ i<j\} }\right) \cdot \\ }
\ensuremath{& \cdot \left( 1 - \ t^{\#\{\l_i\in S(P):\ i\geq j_{P,h+1},\  i\geq j\} }
\right ) \ e^1_{j} }
\end{aligned}\end{equation}}
\bigskip

In particular, let $P$ be a double point. Then $h$ takes only the value $1,$ and $j_{P,1}, \ j_{P,2}$ are the indices of the two lines passing through $P.$  So   formula (\ref{delta2tsep}) becomes 
\bigskip

{\small \begin{equation}\begin{aligned}\ensuremath{&\label{delta2tdoppi}\partial_2(e^2_{P,1})\ =\ 
  t^{\#\{\l_i\in U(P):\ i<j_{P,2}\}} \left(  
1 - t   \right )\ e^1_{j_{P,2}}\  
+    t^{\#\{\l_i\in U(P):\ i<j_{P,1}\}}  \left(  
t - 1   \right )\ e^1_{j_{P,1}}\  +\\ }
\ensuremath{&
+\sum_{  {\tiny
\begin{array}{c}
\l_j\in  C(P)\cap U(P)
\end{array}}
}  t^{\#\{\l_i\in U(P):\ i<j\}  } 
\left ( t-1 \right)^2  \ e^1_{j} }
\end{aligned}\end{equation}}
\bigskip

Since $\partial_2$ is divisible by $t-1$ we can rewrite (\ref{delta2tdoppi}) as
\bigskip
{\small \begin{equation}\label{delta2tfac}\partial_2(e^2_{P,1}) \ =\ (t-1)\ \tilde{\partial}_2(e^2_{P,1})\end{equation}}
where

{\small \begin{equation}\begin{aligned}\ensuremath{& \label{delta2trid}
  \tilde{\partial}_2(e^2_{P,1})\ =\ 
   t^{\#\{\l_i\in U(P):\ i<j_{P,2}\}} \ e^1_{j_{P,2}}\  
-    t^{\#\{\l_i\in U(P):\ i<j_{P,1}\}} \ e^1_{j_{P,1}}\  +\\ }
 \ensuremath{&
+\sum_{  {\tiny
\begin{array}{c}
\l_j\in  C(P)\cap U(P)
\end{array}}
}  t^{\#\{\l_i\in U(P):\ i<j\}  } 
\left ( 1- t \right)  \ e^1_{j} }
\end{aligned}\end{equation}}
\bigskip

\section{A proof in particular cases}\label{sec:proves} We give a proof of conjecture $(1)$ with further hypotheses on $\A.$  

Notice that the rank of $\partial_2$ is $n-1$ (the sum of all rows vanishes). Then  the arrangement has no monodromy iff the only elementary divisor of $\partial_2$ is $\varphi_1:=t-1,$ so $\partial_2$ diagonalizes to $\oplus_{i=1}^{n-1}\ \varphi_1.$ \ This is equivalent to  the reduced boundary $\tilde{\partial}_2$ having an invertible minor of order $n-1.$     

Let $\G$  be the graph of double points. A choice of an admissible coordinate system gives a total ordering on the lines so it induces a labelling, varying between $1$ and $n,$  on the set of vertices $V\Gamma$ of $\Gamma.$ Let $T$ be a spanning tree of $\G$ (with induced labelling on $VT$).

\begin{df}
\label{defcollapse} We say that the induced labelling on $VT=V\G$ is \emph{ very good} (with respect to the given coordinate system) if the sequence $n,\dots,1$ is a \emph{collapsing ordering} on $T.$  In other words, the graph obtained by $T$ by removing all vertices with label $\geq i$ and all edges having both vertices with label $\geq i,$ is a tree, for all $i=n,\dots, 1.$ 

We say that the spanning tree $T$ is  \emph{very good} if there exists an admissible coordinate system such that the induced labelling  on $VT$ is very good (see fig.\ref{fig1}). 
\end{df}  

\begin{rmk} \label{rmk:tree} \begin{enumerate} 
\item A labelling over a spanning tree $T$ gives a collapsing ordering iff for each vertex $v,$ the number of adjacent vertices with lower label is $\leq 1.$
In this case, only the vertex labelled with $1$ has no lower labelled adjacent vertices (by the connectness of $T$).
\item Given a collapsing ordering over $T,$ for each vertex $v$ with label $i_v>1,$ let $\l(v)$ be the edge  which 
connects $v$ with the unique adjacent vertex with lower label; by giving to $\l(v)$ the label $i_v+\frac{1}{2},$ we 
obtain a discrete Morse function on the graph $T$ (see \cite{forman}) with unique critical cell given by the vertex 
with label $1.$  The set of all pairs $(v,\l(v))$ is the acyclic matching which is associated to this Morse function.  
\end{enumerate}
\end{rmk}

Let us indicate by $\G_0$ the linear tree with $n $ vertices: we think as $\G_0$ as a $CW$-decomposition of the real segment $[1,n],$ with vertices $\{j\}, \ j=1,\dots, n,$ and edges the segments $[j,j+1],$ $j=1,\dots,n-1.$   
\begin{df} \label{def:good} We say that a labelling induced by some coordinate system on the tree $T$ is \emph{good} if there exists a permutation $i_1,\dots, i_n$ of $1,\dots , n$ which gives a collapsing sequence both for $T$ and for $\G_0.$ In other words, at each step we always remove either the maximum labelled vertex or the minimum, and this is a collapsing sequence for $T.$  

We say that $T$ is \emph {good} if there exists an admissible coordinate system such that the induced labelling on $VT$ is good (see fig.\ref{fig2}). 
\end{df} 
Notice that a very good labelling  is a good labelling where at each step one removes the maximum vertex.

Consider some arrangement $\A$ with  graph  $\G$  and labels on the vertices which are induced by some coordinate system.   
Notice that changes of coordinates act on the labels by giving all possible {\it cyclic permutations}, which are generated by the transformation  $i\to i+1\ \mbox{mod } n.$  So, given a labelled tree $T,$ checking if $T$ is very good (resp. good) consists in verifying if some cyclic permutation of the labels is very good (resp. good). This property depends not only on the "shape" of the tree, but also on how the lines are disposed in $\R^2$ (the associated oriented matroid). In fact, one can easily find arrangements where some "linear" tree is very good, and others where some linear tree is not good.

\begin{df}\label{defgood} We say that an arrangement $\A$ is \emph{very good} (resp. \emph{good}) if $\G$ is connected and has a very good (resp. good) spanning tree.     
\end{df}

It is not clear if this property is combinatorial, i.e. if it depends only on the lattice.  Of course, $\A$ very good implies $\A$ good.

\bigskip

\begin{teo}\label{teo1}  Let $\A$ be a good arrangement. Then $\A$ is a-monodromic.  
\end{teo}
\begin{proof}  We use induction on the number $n$ of lines,  the claim being trivial for $n=1.$ 

Take a suitable coordinate system as in definition (\ref{defgood}), such that the graph $\G$ has a  spanning tree $T$ with good labelling.  Assume for example that at the first step we remove the last line, so the graph $\G'$  of the arrangement $\A':=\A\setminus \{\l_n\}$ is connected and the  spanning tree $T'$ obtained by removing the vertex $\{\l_n\}$ and the "leaf-edge" \ $(\l_n,\l_j)$\  (for some $j<n$) has a good labelling.    

There are $n-1$ double points which correspond to the edges of $T:$  only one of these is contained in $\l_n,$ namely $\l_n\cap\l_j$ (see remark \ref{rmk:tree}). Let $\Dp:=\{d_1,\dots,d_{n-1}\}$  be the set of such double points, with $d_{n-1}=\l_n\cap\l_j.$  Let also $\Dp ':=\{d_1,\dots,d_{n-2}\},$ which corresponds to the edges of $T'.$   Let $(\CW(\Dp)_*,\partial_*)$  (resp. $(\CW(\Dp')_*,\partial_*')$) be the subcomplex of  $\CW(\A)_*$ generated by the $2$-cells which correspond to $\Dp$ (resp. $\Dp'$  ): then $\CW(\Dp)_2= \oplus_{1\leq i\leq n-1} R e_j\ ,$ and   $\CW(\Dp')_2= \oplus_{1\leq i\leq n-2} R e'_j\ .$   
Notice that, by the explicit formulas given in part \ref{part2}, the component of the boundary $\partial_2(e_j)$ along the $1$-dimensional generator corresponding to $\ell_n$ equals $-\varphi_1$ for $j=n-1,$ and vanishes for  $j=1,\dots,n-2.$  Actually, the natural map taking $e'_j$ into $e_j,$ $j=1,\dots,n-2,$ identifies $\CW(\Dp')_*$ with the sub complex of $\CW(\Dp)_*$ generated by the $e_j$'s, $j=1,\dots,n-2$

\begin{equation}\label{bordoind} \partial_2\ =\ \left[ 
\begin{tabular} {p{3cm} | c}
& \\
\centering{$\partial_2'$} & *\\
 & \\
 \hline
 \centering{$0$} & $-\varphi_1$ 
 \end{tabular}
 \right]\end{equation} 
 \bigskip
 
 Then by induction $\partial'_2$ diagonalizes to $\oplus_{j=1}^{n-2}\ \varphi_1.$ Therefore $\partial_2$ diagonalizes to $\oplus_{j=1}^{n-1}\ \varphi_1,$ which gives the thesis.  
 
 If at the first step we remove the first line, the argument is similar, because $\partial_2(e_j)$ has no non-vanishing components along the generator corresponding to $\l_1.$  
\end{proof}
 \bigskip
 
 Let us consider a different situation.
 
 \begin{df}\label{conjfree} We say that a subset $\Sigma$ of the set of singular points $Sing(\A)$  of the arrangement $\A$ is \emph{conjugate-free} (with respect to a given admissible coordinate system) if $\forall P\in \Sigma$  the set  $U(P)\cap C(P)$ is empty.  
 
 An arrangement $\A$ will be called \emph{conjugate-free} if $\G$ is connected and contains a spanning tree $T$ such that the set of points in $Sing(\A)$ that correspond to the edges $ET$ of $T$ is conjugate-free (see fig.\ref{fig3}).
 \end{df} 
Let  $\Sigma$ be conjugate-free: it follows from formula (\ref{delta2tsep}) that  the boundary of   all generators $e^2_{P,h},\ P\in\Sigma,$   can have non-vanishing components only along the lines which contain $P.$

 \begin{teo}\label{teo:conjfree} Assume that $\A$ is conjugate-free. Then $\A$ is a-monodromic. 
 \end{teo} 
\begin{proof} The sub matrix of $\partial_2$ which corresponds to the double points $ET$ is $\varphi_1$-times the incidence matrix of the tree $T.$  Such matrix is the boundary matrix of the complex which computes the  $\Z$-homology of $T:$ it is a unimodular rank-$(n-1)$ integral matrix (see for example \cite{biggs}). From this the result follows straightforward. \end{proof}
\bigskip

We can have a mixed situation between definitions \ref{defgood} and \ref{conjfree} (see fig.\ref{fig4}). 

\begin{teo} \label{mixed} Assume that $\G$ is connected and contains a spanning tree $T$ which reduces, after a sequence of moves where we remove either the maximum or the minimum labelled vertex, to a subtree $T'$ which is conjugate-free. Then $\A$ is a-monodromic.   
\end{teo}
\begin{proof} The thesis follows easily by induction on the number $n$ of lines. In fact, either $T$ is conjugate-free, and we use theorem \ref{teo:conjfree}, or one of the subtrees $T\setminus \{\l_n\},$ \  $T\setminus \{\l_1\}$ \ satisfies again the hypotheses of the theorem. Assume that it is $T''=T\setminus \{\l_n\}.$ Then the boundary map $\partial_2$ restricted to the $2$-cells corresponding to $ET''$ has a shape similar to (\ref{bordoind}). Therefore by induction we conclude. 
\end{proof} 
\bigskip

Some examples are given in section \ref{sec:ex}.

\begin{rmk} In all the theorems in this part, we have proven a stronger result: namely,  the subcomplex spanned  by the generators corresponding to the double points is  a-monodromic.
\end{rmk}

\section{Examples}\label{sec:ex}

\newcommand{\ut}{{\underline{t}}}

 In this section we give examples corresponding to the various definitions of section \ref{sec:proves}. We include the computations of the local homology of the complements. 
 
In figure \ref{fig1} we show an arrangement having a very good tree (def \ref{defcollapse}) and the associated sequence of contractions.
\begin{figure}[h]
\centering
  \subfigure
    {\includegraphics[width=0.5\textwidth]{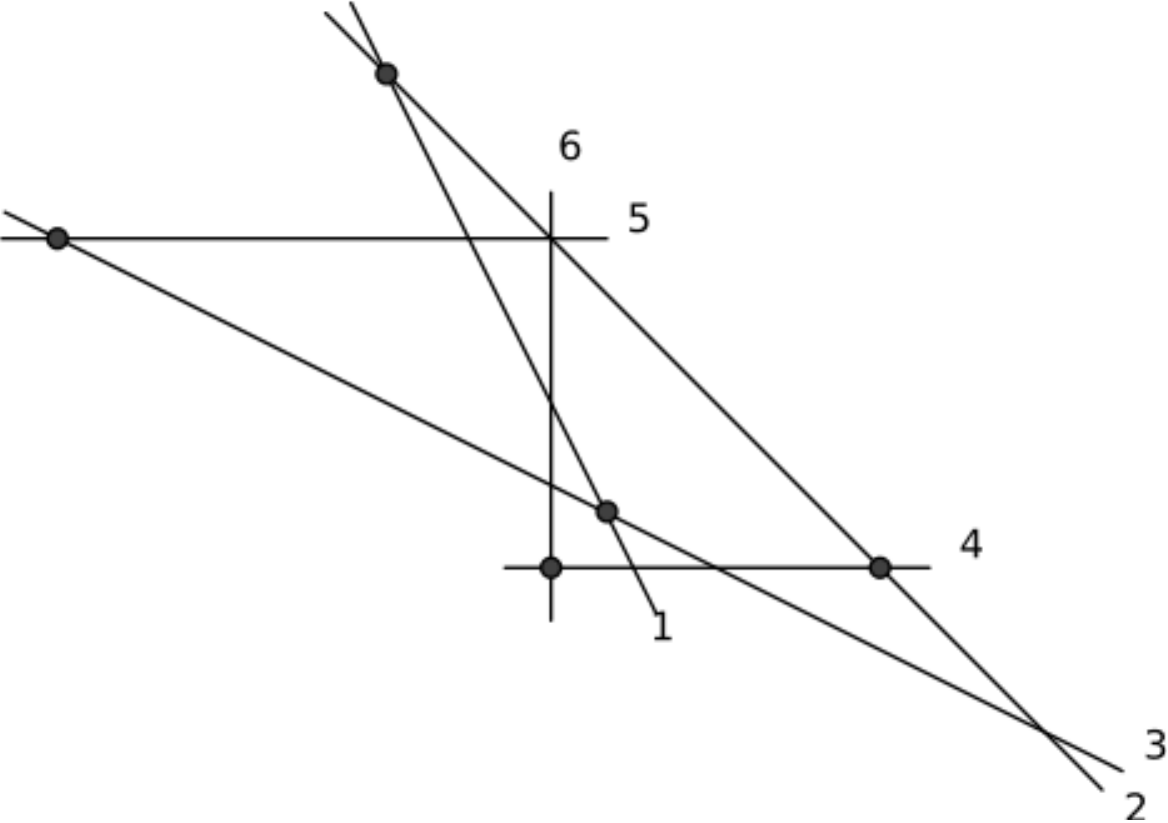}}
  \subfigure
    {\includegraphics[width=0.5\textwidth]{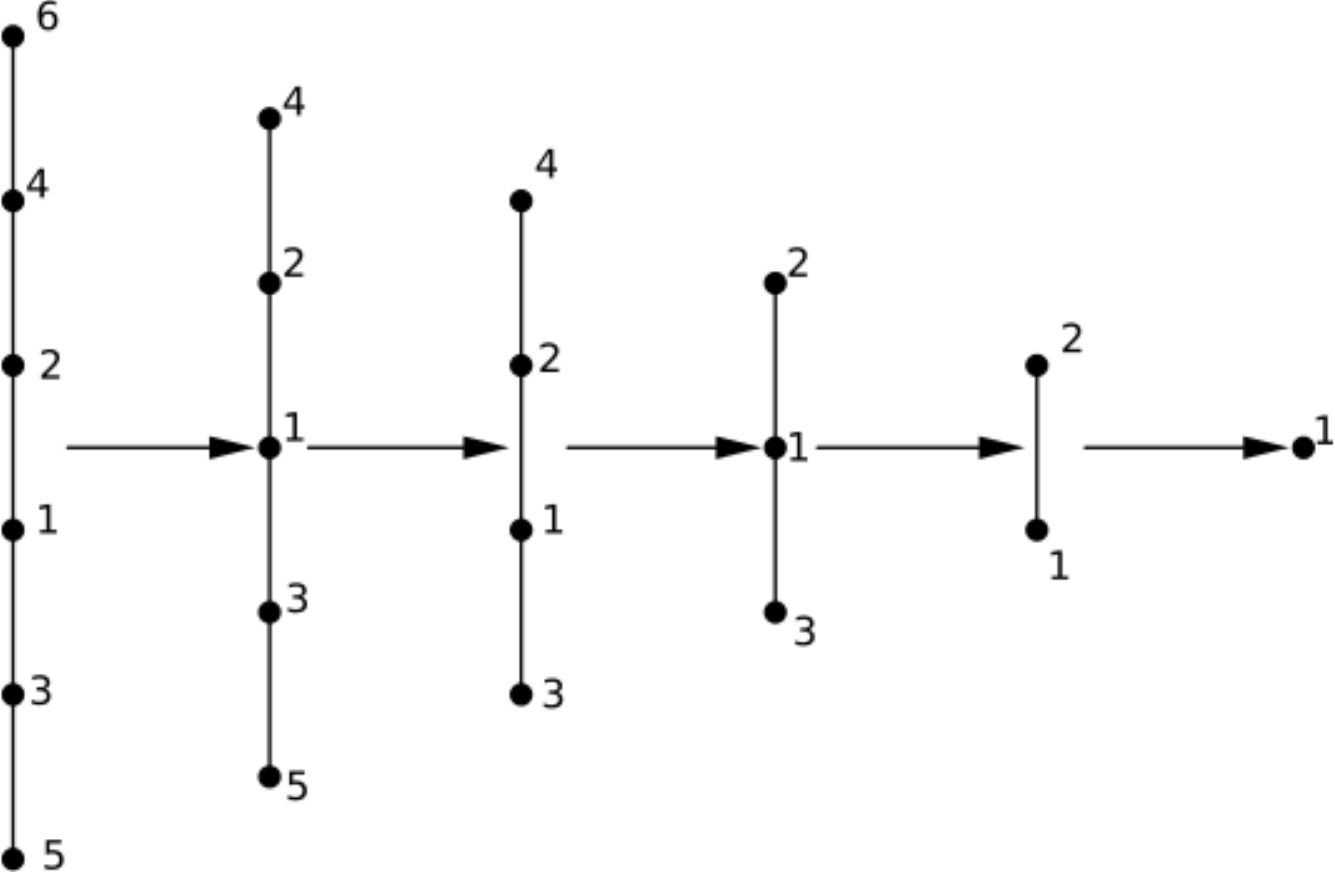}}
    \caption{}\label{fig1}
\end{figure}

\newpage

 In figure \ref{fig2} an arrangement with a good tree is given (def \ref{def:good}) together with its sequence of contractions.

\begin{figure}[h]
\centering
  \subfigure
    {\includegraphics[width=0.5\textwidth]{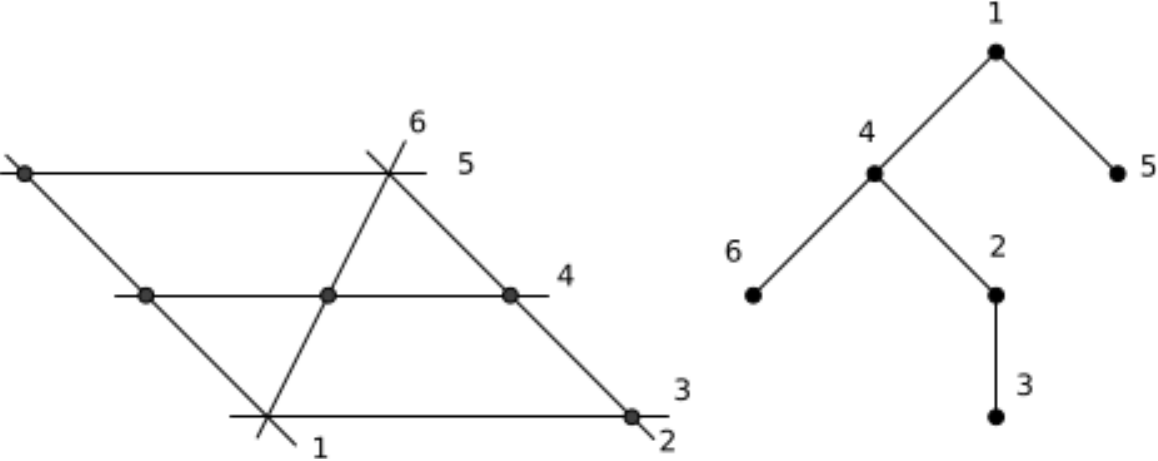}}
  \subfigure
   {\includegraphics[width=0.5\textwidth]{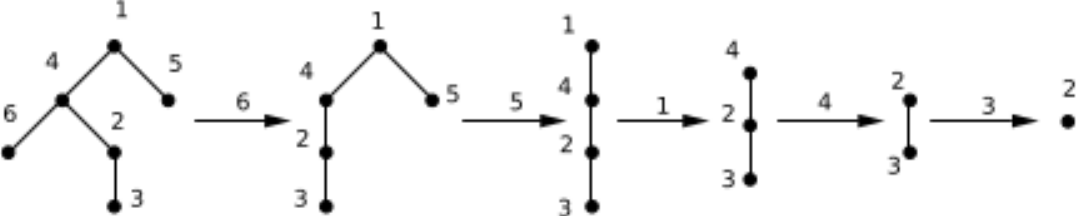}}
   \caption{}
   \label{fig2}
\end{figure}
\newpage
An arrangement having a tree which is both conjugate free (see definition \ref{conjfree}) and good is depicted in figure 
\ref{fig3}

\begin{figure}[h]
\centering
  \subfigure
    {\includegraphics[width=0.8\textwidth]{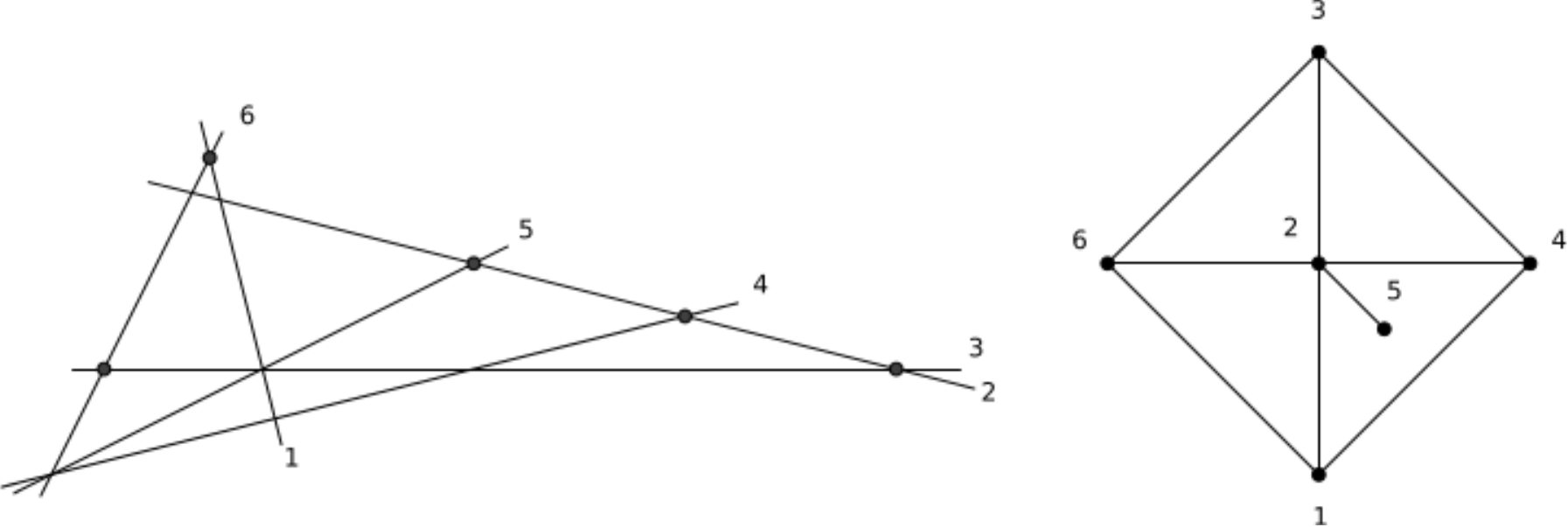}}
  \subfigure
    {\includegraphics[width=0.2\textwidth]{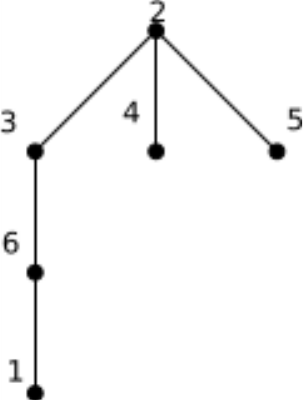}}
    \caption{}\label{fig3}
\end{figure}
\newpage
In figure \ref{fig4} an arrangement with 
a tree which after  2 admissible contractions  becomes  conjugate free is shown (see thm.\ref{mixed}).
\begin{figure}[h]
 \centering
   \subfigure
     {\includegraphics[width=0.7\textwidth]{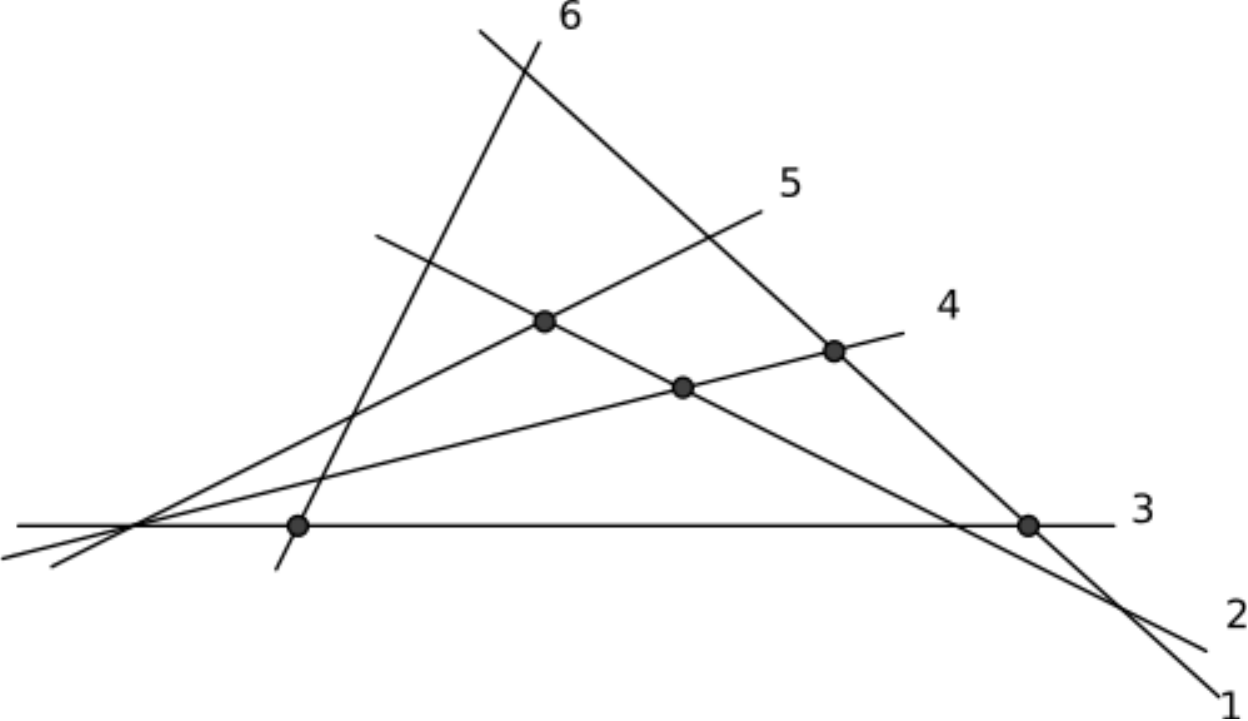}}
   \subfigure
        {\includegraphics[width=0.4\textwidth]{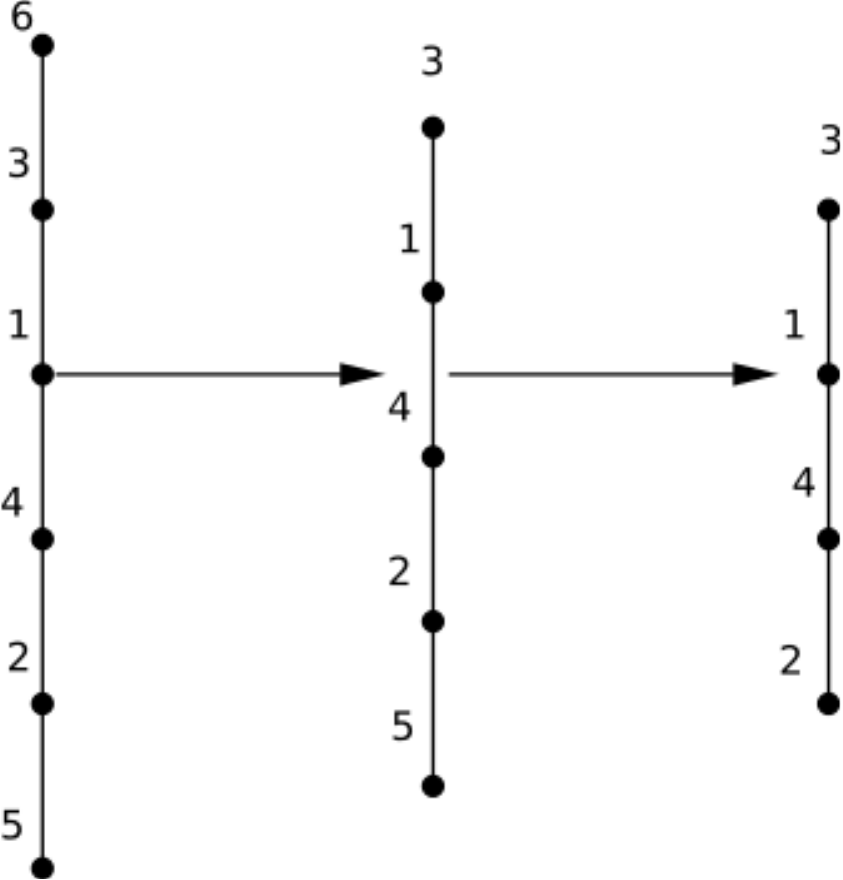}}
\caption{}\label{fig4}
\end{figure}
\newpage
Next we give some example of arrangements with non trivial monodromy. Notice that the graph of double points is disconnected in these cases. 

Notice also that in the first two examples one has non-trivial monodromy both for the given affine arrangement and its conifed arrangement in $\C^3;$ in the last example, the given affine arrangement has non trivial monodromy while its conification is a-monodromic.
\begin{figure}[h]
\centering
  {\includegraphics[width=0.5\textwidth]{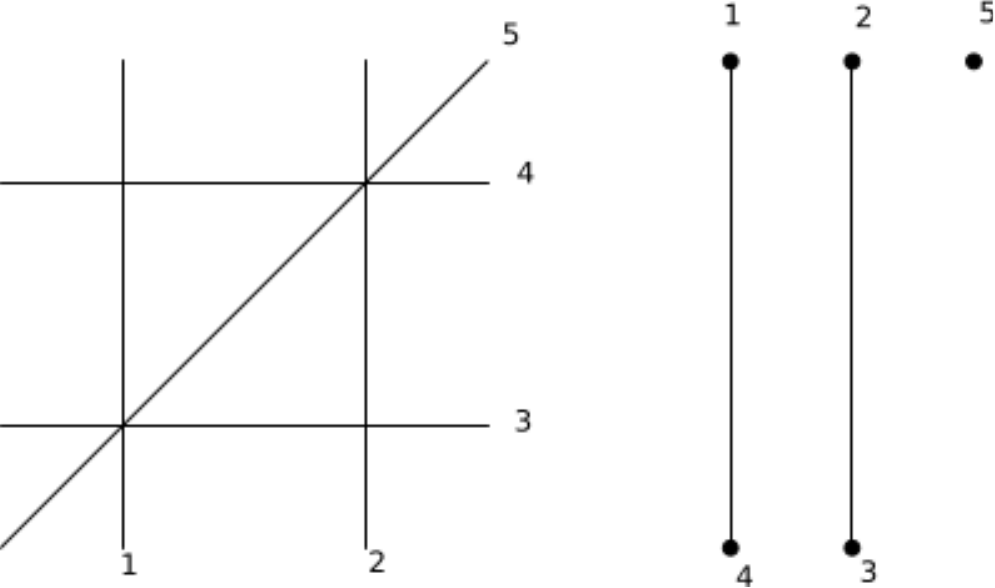}}
\caption{deconed A3 arrangement} 
\label{fig5}
\end{figure}
\\
\begin{align*}
 \ensuremath{& H_1(\M(\A),\Q[t^{\pm 
1}])\simeq\left(\frac{\Q[t^{\pm 1}]}{(t-1)}\right)^3\oplus\frac{\Q[t^{\pm 1}]}{(t^3-1)}}
\end{align*}
\\
\begin{figure}[h]
 \centering
 {\includegraphics[width=0.8\textwidth]{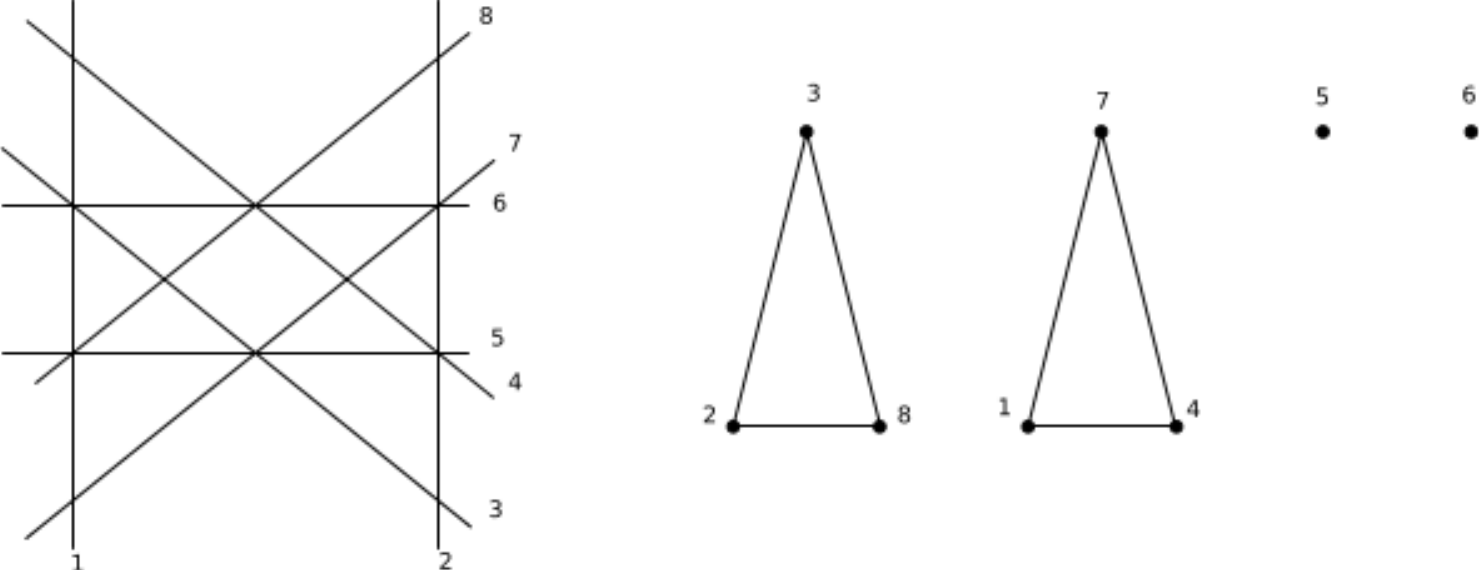}}
\caption{deconed Pappus arrangement}
\label{fig6}
\end{figure}
\\
\begin{align*}
 \ensuremath{& H_1(\M(\A),\Q[t^{\pm 
1}])\simeq\left(\frac{\Q[t^{\pm 1}]}{(t-1)}\right)^6\oplus\frac{\Q[t^{\pm 1}]}{(t^3-1)}}
\end{align*}
\\
\begin{figure}[h]
 \centering
{\includegraphics[width=0.8\textwidth]{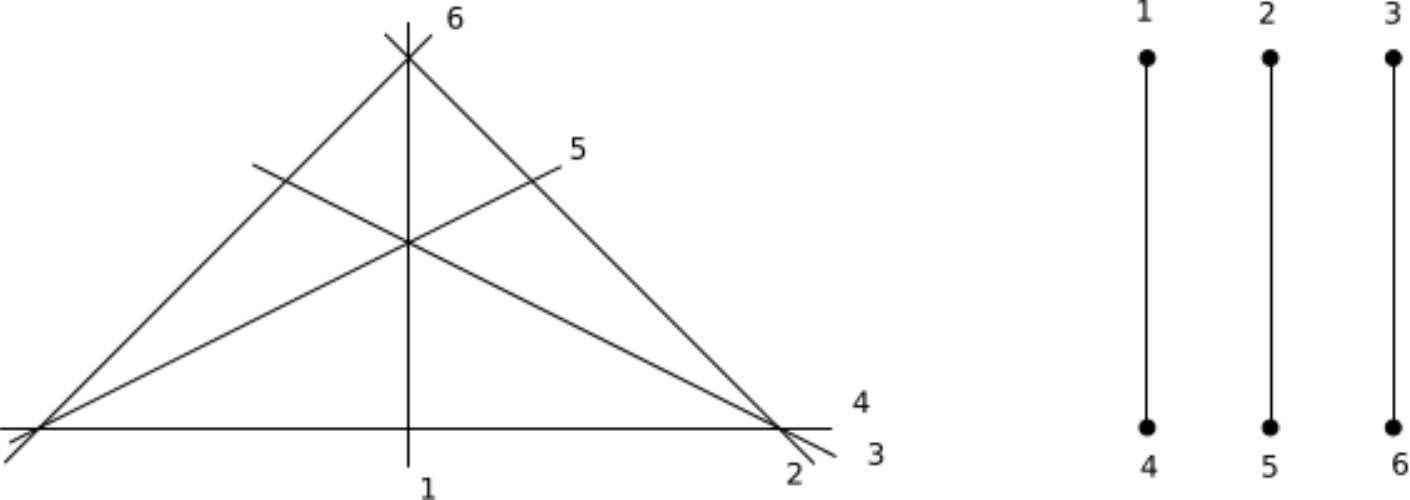}}
\caption{The "complete triangle" has non-trivial monodromy but its conification is a-monodromic} \label{completetriangle}
\end{figure}
\\
\begin{align*}
 \ensuremath{& H_1(\M(\A),\Q[t^{\pm 
1}])\simeq\left(\frac{\Q[t^{\pm 1}]}{(t-1)}\right)^4\oplus\frac{\Q[t^{\pm 1}]}{(t^3-1)}}
\end{align*}
\bigskip

We focus here on the structure of the fundamental groups of the above examples, in particular in case of a-monodromic arrangements.
\medskip

\ni For arrangement in fig.$\ref{fig1}$: after taking line $5$ to infinity we obtain an affine arrangement having only double points with two pairs of parallel lines, namely (the new) lines $2,\ 6$ and $4,\ \infty.$ Therefore 
$$\pi_1(\MA)=\Z\times \Z \times F_2 \times F_2.$$
\medskip

\ni We consider arrangement in fig.\ref{fig2} and in fig.\ref{fig5} together.  The deconed $A3$ arrangement in fig.\ref{fig5}  is a well known $K(\pi,1)$-arrangement: the fundamental group of the complement is the \emph{pure braid group} $P_4$ in $4$ strands. Notice that the projection onto the $y$ coordinate fibers the complement over $\C\setminus \{2\ pts\}$ with fiber $\C\setminus \{3\ pts\}.$  It is well known that this fibering is not trivial and we obtain a semi-direct product decomposition
$$\pi_1(\MA)=F_3 \rtimes F_2.$$

\ni The same projection gives a fibering of the complement of the arrangement in fig.\ref{fig2} over \ 
$\C\setminus \{3\ pts\}$ with fiber $\C\setminus \{3\ pts\}.$ Notice that this is also a non-trivial fibering, so 
we have a semi-direct decomposition
$$\pi_1(\MA)=F_3 \rtimes F_3.$$
In particular, we have an a-monodromic arrangement such that the fundamental group of the complement is not a direct product of free groups.
\medskip

\ni In the arrangement of fig.\ref{fig3} the line at infinity is transverse to the other lines. If we take line $5$ at infinity we get an affine arrangement with only double points, with two pairs of parallel lines $1,\ 3$ and $4,\ 6.$ Therefore  we obtain a decomposition of  $\pi_1(\MA)$ as in case of fig.\ref{fig1}.   
\medskip

The arrangement of fig.\ref{fig4} has only one triple point. By taking line $5$ to infinity we get an affine arrangement with only double points and one pair of parallel lines $3,\ 4.$  Therefore
$$\pi_1(\MA)=\Z^4 \times F_2.$$
\medskip

The \emph{complete triangle} in fig.\ref{completetriangle} becomes, after taking any line at infinity,  the affine arrangement $\A'$ which is obtained from the $A3$ deconed arrangement in fig.\ref{fig5} by adding  one more line $\overline{\elle}$ which is transverse to all the others. Therefore 
$$\pi_1(\MA)=\Z\times (F_3 \rtimes F_2).$$
\begin{rmk} It turns out that the arrangement  $\A'$ is a-monodromic. This is not a contradiction: in fact, one is considering two different local systems on $\M(\A')=\MA.$  The a-monodromic one associates to an elementary loop around $\overline{\elle}$ the $t-$multiplication.  This is different from the one obtained by exchanging one of the affine lines of the arrangement $\A$ in fig.\ref{completetriangle} with the infinity line. In this case we should associate to an elementary loop around $\overline{\elle}$ the $t^6-$multiplication, and then apply formula (\ref{delta2gen}).
\end{rmk}


\section{Free calculus}\label{sec:free}
In this section we reformulate our conjecture in terms of  Fox calculus.

Let $\A=\{l_1,\dots,l_n\}$ be as above; if we denote by $\beta_i$ an elementary loop around $l_i$ we have that the fundamental group $\pi_1(\M(\A)) $ is 
generated by $\beta_1,\dots,\beta_n$ and a presentation of this group is given for ex. in \cite{salvetti}. Let $R=\Q[t^{\pm1}]$ 
be as above with the given structure of $\pi_1(\M(\A))-$module.

We denote by $F_n$ the free group generated by $\beta_1,\dots,\beta_n.$ Let $\varphi: F_n \to\ <t>$ 
be the group homomorphism defined by $\varphi(\beta_i)=t$ for every $1\leq i \leq n$ where $<t>$ is the multiplicative 
subgroup of $R$ generated by $t$.  As in \cite{birman}, if $w$ is a word in the $\beta_j$'s, we use the notation $w^{\varphi}$ for $\varphi(w).$

Consider the algebraic complex which computes the local homology of $\M(\A)$ introduced in section \ref{part2}.  The following remark is crucial for the rest of this section: if $e^2_{P,j}$ is a two-dimensional generator corresponding to a  two-cell which is attached along the word  $w$ in the $\beta_j$'s, then $\left(\frac{\partial 
w}{\partial \beta_i}\right)^\varphi$ is the coefficient of $e^1_i$ of the border of $e^2_{P,j}$.  This easily follows  from the combinatorial calculation of local system homology.

Let  $l:F_n\longrightarrow \Z$ be the length function, given by
$$l(\beta_{i_1}^{\epsilon_1}\cdots \beta_{i_r}^{\epsilon_r} )=\sum_{k=1}^r \epsilon_k.$$
Then $\text{ker}\varphi$ is the normal subgroup of $F_n$ given by the words of lenght $0$.

 Each relation in the fundamental group $\pi_1(\M(\A))$ is a commutator (cfr. \cite{salvetti},\cite{elgate}), so it lies in 
$\text{ker}\varphi.$ 
 So, in the sequel, we consider only words in 
$\text{ker}\varphi$.

\begin{rmk}\label{rmk:eqconj1} The arrangement $\A$ is a-monodromic iff (by definition) the $\mathbb Q[t^{\pm 1}]$-module generated by $\partial_2(e^2_j),\ j=1,\dots,\nu_2,$ equals \ $(t-1)\text{ker}\partial_1.$\  One has: 
$\text{ker}\partial_1= \{ \sum_{j=1}^n\ x_j\ e^1_j\ :\ \sum_{j=1}^n\ x_j =0\}. $
\end{rmk}

 \noindent Let $R_j,\ j=1\dots,\nu_2,$ be a complete set of relations in $\pi_1(\M(\A))$.  We use now $e^2_j$ to indicate the two-dimensional generator  corresponding to a two-cell which is attached along the word $R_j.$ Then the boundary of $e^2_j$ is given by

 \begin{equation}\label{boundary2}
 \partial_2(e^2_j)=\ \sum_{i=1}^n\ \left(\frac{\partial R_j }{\partial \beta_i}\right)^\varphi\ e^1_i,\quad j=1,\dots,\nu_2.
 \end{equation}
Then by remark \ref{rmk:eqconj1} $\A$ is a-monodromic iff  each element of the shape
\begin{equation}\label{polinomi} P(t):= (1-t)\sum_{i=1}^n\ P_i(t)\ e^1_i,\quad \sum_{i=1}^nP_i(t)=0,\quad (P_i(t)\in\mathbb Q[t^{\pm 1}], i=1\dots,n)
\end{equation}
is a linear combination with coefficients in $\mathbb Q[t^{\pm 1}]$ of  the elements in (\ref{boundary2}), i.e.:
\begin{equation}\label{equazione1} P(t)=\sum_{j=1}^{\nu_2} Q_j(t) \partial_2(e^2_j),\quad Q_j(t)\in\Q[t^{\pm 1}].
\end{equation}

It is natural to ask about solutions with coefficients in $\mathbb Z[t^{\pm 1}]$ instead of $\mathbb Q[t^{\pm 1}].$
We say that $\A$ is a-monodromic over $\mathbb Z$ if  there is solution of (\ref{equazione1}) over $\mathbb Z[t^{\pm 1}]$ (when all the $P_i(t)$'s in (\ref{polinomi}) are in $\Z[t^{\pm 1}]$).
\bigskip

\begin{teo} \label{teo:suZ} The arrangement $\A$ is a-monodromic over $\Z$ iff  $\pi_1(\M(\A))$ is commutative modulo $[\text{ker}\varphi,\text{ker}\varphi].$ More precisely, $\A$ is a-monodromic over $\Z$ iff  
$$[F_n,F_n] = N [\kf,\kf], $$
 where $N$ is the normal subgroup generated by the relations $R_j$'s .
\end{teo}
\proof  A set of generators for $(t-1)\text{ker}\partial_1$ as $\Zt$-modulo is given by all elements of the type
$$ P_{rs}:=(1-t)(e^1_r-e^1_s)\ ,\quad  r\not= s.$$
Such an element can be re-written in the form (\ref{boundary2}) as
$$P_{rs}=\  \sum_{i=1}^n\ \left(\frac{\partial [\beta_r,\beta_s] }{\partial \beta_i}\right)^\varphi\ e^1_i$$
where $[\beta_r,\beta_s]=\beta_r\beta_s\beta_r^{-1}\beta_s^{-1}.$ Now there exists an expression as in (\ref{equazione1}) for $P_{rs},$ with all $Q_j(t)\in \Zt$ iff 
\begin{equation}\label{equazione2}\left(\frac{\partial [\beta_r,\beta_s] }{\partial \beta_i}\right)^\varphi\ =\  
\left(\frac{\partial \prod_{j=1}^{\nu_2}\ R_j^{Q_j(\beta_1)}}{\partial \beta_i}\right)^\varphi,\quad i=1,\dots,n.
\end{equation}
Here $Q_j(\beta_1)\in \Z[F_n]$  is obtained by substituting $t$ with $\beta_1$ (any word of length one would give the same here). Moreover, for $R, w$ any words in  $\kf$ we set $R^w:=wRw^{-1},$ and for $a\in \Z$ we set
$R^{aw}:=R^w\dots R^w$\  ($a$ factors) if $a>0$ and $(R^{-1})^w\dots(R^{-1})^w$ \ ($|a|$ factors) for $a<0.$
Also, we set $R^{aw+bu}:=R^{aw}R^{bu}.$  Then equalities (\ref{equazione2}) come from standard Fox calculus.

Then from Blanchfield theorem (see \cite{birman}, chap. 3) it follows that 
$$[F_n,F_n] \subset N [\kf,\kf].$$
The opposite inclusion follows because, as we said before remark \ref{rmk:eqconj1}, for any arrangement one has $N\subset [F_n,F_n].$  \qed
\bigskip 

\begin{rmk} Condition in theorem \ref{teo:suZ} is equivalent to the equality  
$$\frac{F_n}{N[\kf,\kf]} = \frac{F_n}{[F_n,F_n]}= H_1(\M(\A);\Z).$$ 
\end{rmk}
\bigskip

Since $\kf\supset [F_n,F_n],$ so $[\kf,\kf]\supset [[F_n,F_n],[F_n,F_n]],$ it follows immediately from theorem \ref{teo:suZ} 
\bigskip

\begin{cor}\label{perfect} Assume that {\small $\pi_1(\M(\A))/ \pi_1(\M(\A))^{(2)}$} is abelian (iff \ {\small $\pi_1(\M(\A))^{(1)}=\pi_1(\M(\A))^{(2)}$}) where  ${\small \pi_1(\M(\A))^{(i)}}$ is the $i$-th element of the derived series of ${\small \pi_1(\M(\A)),\ i\geq 0.}$ Then $\A$ is a-monodromic over $\Z.$
\end{cor}  \qed
\bigskip

%

\ni Condition of corollary \ref{perfect} corresponds to the vanishing of the so called \emph{Alexander invariant} of  $\pi_1(\MA).$
 
As a subgroup of the free group $F_n,$ the group $\text{ker}\varphi$ is a free group   We use the Reidemeister-Schreier method to write an explicit list of generators.
Notice that for any fixed $1\leq j\leq n$, the set $\{\beta_j^k : k\in\Z\}$ is a Schreier right coset 
representative system for $F_n/\text{ker}\varphi$. Denote briefly by $s_{k,i}$ the element 
$s_{\beta_j^k,\beta_i}=\beta_j^k\beta_i\overline{(\beta_j^k\beta_i)}^{-1}=\beta_j^k\beta_i\beta_j^{-(k+1)}.$ Then
$$\text{ker}\varphi=<\{s_{k,i}:1\leq i\leq n, \ k\in\Z \}; s_{k,i}>$$
where $s_{k,i}$ is a relation if and only if $\beta_j^k\beta_i$ is freely equal to $\beta_j^{k+1};$  this happens if and only if $i=j$. So $\text{ker}\varphi$ is the free group generated by $\{s_{k,i}:k\in\Z,1\leq i\leq n, i\neq j\}.$ Its abelianization 
$$\text{ab}\left(\text{ker}\varphi\right)=\text{ker}\varphi/[\text{ker}\varphi,\text{ker}\varphi]$$
 is the free abelian group on the classes $\overline{s}_{k,i}$ of the generators $s_{k,i}$'s, $i\not= j.$
 Let
$$\text{ab}:\text{ker}\varphi\too \text{ab}\left(\text{ker}\varphi\right)$$
be the abelianization homomorphism.

Now we define the automorphism $\sigma$ of $\text{ker}\varphi$ by
$$ \sigma(s_{k,i})=s_{k+1,i}$$
which passes to the quotient, so it defines an automorphism (call it again $\sigma$) of $\text{ab}\left(\text{ker}\varphi\right)$. Therefore we may view 
$\text{ab}\left(\text{ker}\varphi\right)$ as a finitely genereted free $\Z[\sigma^{\pm 1}]$-module, with basis  
$\overline{s}_{0,i}$ with $1\leq i\leq n$ and $i\neq j$.


%



In this language  theorem (\ref{teo:suZ}) translates as
\bigskip

\begin{teo}\label{teo:congliesse} The arrangement $\A$ is a-monodromic over $\Z$ iff the submodule $(1-\sigma)\text{ab}(\kf)$ of $\text{ab}(\kf)$ is generated by $\text{ab}(R_j),$ $j=1,\dots,\nu_2,$ as $\Z[\sigma^{\pm 1}]$-module.
\end{teo} \qed
\bigskip  

Of course, one can give a  conjecture holding over $\Z.$
\bigskip

 \noindent$\begin{array}{cl} \label{conj2}
{\bf  Conjecture\ 3}: &  {\mbox{\emph{ \ Assume that $\G$ is connected; then $\A$ is a-monodromic over $\Z.$   }}}
 \end{array}$

 \bigskip
 Conjecture $3$ clearly implies conjectures $1$ and $2.$
Our experiments agree with this stronger conjecture.


We give explicit computations for the arrangements in fig.\ref{fig1} and fig.\ref{fig5}.  The $\Z[\sigma^{\pm 1}]$-module\  $(1-\sigma)\text{ab}(\kf)$ is generated by $\{(1-\sigma) \Os_{0,i},\ , i\not= j\}.$  We choose  $j$ as the last index in the natural ordering. All abelianized relations are divisible by $(1-\sigma),$ so we just divide everything by $1-\sigma$ and verify that $\text{ab}(\kf)$ is generated by \ $\text{ab}(R_j)/(1-\sigma).$

For the arrangement in fig.\ref{fig1} we have to rewrite $13$ relations coming from $11$ double points and $1$ triple point. After abelianization we obtain:
\bigskip


\begin{tabular} {lll} (a)\ $\Os_{0,2}-\Os_{0,3};$ &(b)\ $\Os_{0,2}-\Os_{0,4};$ &(c)\ $\Os_{0,3}-\Os_{0,4};$ \\
 $(d)\ \Os_{0,1}-\Os_{0,4};$ &
 $(e)\ \Os_{0,1}-\Os_{0,3};$ & 
 $(f)\ \sigma\Os_{0,2}+\Os_{0,5};$ \\   
 $(g)\ (1+\sigma)\Os_{0,2}-\sigma\Os_{0,5};$ &
$(h)\ \Os_{0,1}+\sigma^{-1}(1-\sigma)\Os_{0,2}-\sigma^{-2}(1+\sigma)\Os_{0,5};$ & 
$(i)\ \Os_{0,3}+(\sigma^{-1}-1)\Os_{0,5};$ \\
 $(j)\ \Os_{0,4}+(\sigma^{-1}-1)\Os_{0,5};$ &  
$(k)\ \Os_{0,1}+(\sigma^{-1}-1)\Os_{0,2}-\sigma^{-1}\Os_{0,5};$ & 
$(l)\ \Os_{0,1}-\Os_{0,2};$\\ 
$(m)\ \Os_{0,3}-\Os_{0,5}$
\end{tabular}
\bigskip

The generator $\Os_{0,5}$ is obtained as\  $\sigma(\ (i) - (m)\ ).$  From $\Os_{0,5}$ we obtain in sequence all the other generators $\Os_{0,3}, \ \Os_{0,1}, \ \Os_{0,4},\ \Os_{0,2}.$  According to theorem \ref{teo:congliesse} 
this gives the a-monodromicity of the arrangement in fig.\ref{fig1}.

For the arrangement $A3\ deconed$ in fig.\ref{fig5} we have to rewrite two relations for each triple point and one relation for each double point. Their abelianization is given by:
\bigskip

\begin{tabular}{ll} (a) $\Os_{0,2}-\Os_{0,3};$ & 
(b) $\sigma\Os_{0,2}+\Os_{0,4};$ \\
(c) $(\sigma+1)\Os_{0,2}-\sigma \Os_{0,4};$&
(d) $\sigma \Os_{0,1}+(1-\sigma)\Os_{0,2}+(\sigma^{-1})\Os_{0,3}+(\sigma^{-2}-1)\Os_{0,4};$ \\
(e) $\Os_{0,1}+(\sigma^{-1}-1)\Os_{0,2}-(\sigma^{-1})\Os_{0,4}$ &
(f) $(\sigma+1) \Os_{0,1}+(\sigma^{-1}-\sigma)\Os_{0,2}-\Os_{0,3}+(\sigma^{-2}-\sigma^{-1})\Os_{0,4}$ 
\end{tabular}
\bigskip

We perform the following base changes: 
\bigskip

\begin{tabular}{ll} 
(a') = (a);& (b')=(b) - $\sigma$ (a);\\
 (c')= (c) - (b) - (a) &
(d')= (d) - $\sigma^{-2} $(b) + $\sigma^{-1}$ (a) - $\sigma$ (e)\\
  (e') = (e) + $\sigma^{-1}$ (b) - $\sigma^{-1}$ (a) &
(f') = (f) - ($\sigma^{-2}$+$\sigma^{-1}+1$) (b) + $\sigma$ (a) + $\sigma^{-1}$ (c) - $(\sigma+1)$ (e)
\end{tabular}
\bigskip

and
\bigskip

\begin{tabular}{lll} 
$\Os_{0,1}'\ =\ \Os_{0,1} + \sigma^{-1}\ \Os_{0,3}$; && $\Os_{0,2}'\ =\ \Os_{0,2}\ -\ \Os_{0,3}$; \\
$\Os_{0,3}'\ =\  \Os_{0,3}$; && $\Os_{0,4}'\ =\ \Os_{0,4}\ +\ \sigma\ \Os_{0,3}$
\end{tabular}
\bigskip

\noindent It is straighforward to verify, after these changes, that the submodule $M$ generated by 
$<ab(R_j):\ j=1,\dots,6>$ \  equals  
$$<\Os_{0,1}',\Os_{0,2}', (1+\sigma+\sigma^2)\Os_{0,3}',\Os_{0,4}'>.$$
So $M\subsetneq (1-\sigma)\text{ab}(\kf),$ in accordance with theorem \ref{teo:congliesse}.

\section{Further characterizations}

In this section we give a more intrinsic picture.

Let  $\AT=\{H_0,H_1,\dots,H_n\}$ be the conified arrangement in $\C^3.$  
The fundamental group
$${\mathbf G}=\pi_1(\MT)\quad (=\pi_1(\MA)\times\Z)$$
is generated by elementary loops $\beta_0,\dots,\beta_n$ around the hyperplanes.

Let 
$$\mathbf{F=F_{n+1}}[\beta_0,\dots, \beta_n]$$
be the free group and
${\mathbf N}$
be the normal subgroup generated by the relations, so we have a presentation

$$1\stackrel{}{\too} \bN \stackrel{}{\too} \bF \stackrel{\pi}{\too} \bG \stackrel{}{\too} 1$$

The length map $\varphi:\bF\to <t>\cong\Z$ factors through $\pi$ by a map
$$\psi:\bG\to \Z.$$

\ni Next, $\psi$ \ factorizes through the abelianization  

$$\frac{\bG}{[\bG,\bG]}\cong H_1(\MT;\Z)\cong \Z^{n+1}\cong\frac{\bF}{[\bF,\bF]} .$$

Let now 
$$\bK=\ ker{\psi}$$ 
so we have

\begin{equation}\label{milnor0}1\stackrel{}{\too} \bK \stackrel{}{\too} \bG \stackrel{\psi}{\too} \Z \stackrel{}{\too} 1\end{equation}
\medskip

\noindent and $\psi$ factorizes through  
$$\bG\stackrel{ab}{\too} \frac{\bG}{[\bG,\bG]}\cong\Z^{n+1} \stackrel{\lambda}{\too} \Z$$
\bs

We have  a commutative diagram:
\bs

\begin{center}
\begin{equation}\label{diagram1}
\begin{tabular}{c}
\xymatrix @R=2.5pc @C=2.5pc {
 & & 1 \ar[d] & & \\
 & & \bN \ar[d] & &1 \\
1\ar[r] & ker(\varphi) \ar[r] &\bF  \ar[r]^(.4){\varphi} \ar[d]_(.4){\pi} & \Z \ar[r]\ar[ru] &1 \\
1\ar[r] & [\bG,\bG] \ar[r] \ar[d] &\bG  \ar[r]^(.4){ab} \ar[ru]^(.45){\psi} \ar[d]& \frac{\bG}{[\bG,\bG]} \ar[r] \ar[u]_(.6){\lambda} &1 \\
 & \bK \ar[ru] &  1  & & \\
 1 \ar[ru] & & & & 
 } 
\end{tabular}
\end{equation}
\end{center}

\bs

\begin{rmk} One has
$$ker(\lambda)\ =\ \frac{\bK}{[\bG,\bG]}$$
so 
$$\frac{\bK}{[\bG,\bG]}\  \cong \Z^n$$
\end{rmk}
\bs

\ni Therefore diagram (\ref{diagram1}) extends to

\begin{center}
\begin{equation}\label{diagram2}
\begin{tabular}{c}
\xymatrix @R=2.5pc @C=2.5pc {
 & & 1 \ar[d] & & \\
 & & \bN \ar[d] & 1 & 1 \\
1\ar[r] & ker(\varphi) \ar[r] &\bF  \ar[r]^(.4){\varphi} \ar[d]_(.4){\pi} & \Z \ar[r] \ar[u] \ar[ru]&1 \\
1\ar[r] & [\bG,\bG] \ar[r] \ar[d] &\bG  \ar[r]^(.4){ab} \ar[ru]^(.45){\psi} \ar[d]& \frac{\bG}{[\bG,\bG]} \ar[r] \ar[u]_(.6){\lambda} &1 \\
 & \bK \ar[ru] &  1  & \frac{\bK}{[\bG,\bG]} \ar[u] & \\
 1 \ar[ru] & & & 1 \ar[u] & 
 } 
\end{tabular}
\end{equation}
\end{center}
\vskip1.5cm

Recall the $\Z[t^{\pm 1}]-$module isomorphism:

\begin{equation}\label{isomorf1} H_1(\bG;\Z[t^{\pm1}])\ \cong\  H_1(F;\Z)\end{equation}
\medskip

\ni where $F$ is the Milnor fibre, and (by Shapiro Lemma):

\begin{equation}\label{isomorf2}  H_1(F;\Z) \cong H_1(\bK;\Z) = \frac{\bK}{[\bK,\bK]}\end{equation}
\medskip

\begin{rmk}
There is an exact sequence
\bs

\begin{equation}\label{milnor1}1\stackrel{}{\too} \frac{[\bG,\bG]}{[\bK,\bK]} \stackrel{}{\too} \frac{\bK}{[\bK,\bK]} \stackrel{}{\too} \frac{\bK}{[\bG,\bG]}\cong\Z^n \stackrel{}{\too} 1
\end{equation}

\end{rmk}
\bs

From the definition before thm. \ref{teo:suZ} one has
\bs

\begin{lem}
The arrangement $\AT$ is a-monodromic over $\Z$ iff 
\bs

$$H_1(F;\Z)\cong \Z^n$$
\end{lem} 

\bs

\noindent It follows
\bs

\begin{teo}
The arrangement $\AT$ is a-monodromic over $\Z$ iff 
\begin{equation}\label{condition1}\frac{[\bG,\bG]}{[\bK,\bK]}\ =\ 0 \end{equation}
\end{teo}
\proof It immediately follows from sequence \ref{milnor1} and from the property that a surjective endomorphism of a finitely generated free abelian group is an isomorphism.
\qed
\bs

Since
$$\bK\supset [\bG,\bG]$$
it follows immediately (see \ref{perfect})
\bigskip

\begin{cor}
Assume 

$$\bG^{(1)}=[\bG,\bG]\ =\  \bG^{(2)}=[[\bG,\bG],[\bG,\bG]]$$

Then the arrangement $\AT$ is a-monodromic.
\end{cor}
\bigskip

We also have:  
\bigskip

\begin{cor}\label{cor:central} Let $\bG$ have a central element of length $1.$ 
Then the arrangement $\AT$ is a-monodromic.
\end{cor}
\proof Let $\gamma\in\bG$ be a central element of length $1.$ From sequence (\ref{milnor0}) the group splits as a direct product
$$\bG\cong \bK\times \Z$$
where $\Z=<\gamma>.$ Therefore clearly  $[\bG,\bG]=[\bK,\bK].$ \qed 
\bs  

\noindent An example of corollary is when one of the generators $\beta_j$ commutes with all the others, i.e. one hyperplane is transversal to the others. So, we re-find in this way a well-known fact.

Consider again the exact sequence (\ref{milnor1}). Remind that
the arrangement $\AT$ is a-monodromic (over $\Q$) iff \ $H_1(F;\Q)\cong \Q^n.$
By tensoring sequence (\ref{milnor1}) by $\Q$ we obtain

\begin{teo}
The arrangement $\AT$ is a-monodromic (over $\Q$) iff 
$$\frac{[\bG,\bG]}{[\bK,\bK]}\otimes\Q =\ 0 $$
\end{teo}

\begin{rmk} All remarkable questions about the $H_1$ of the  Milnor fibre $F$ are actually questions about the group 
$$\frac{[\bG,\bG]}{[\bK,\bK]}$$
In particular:
\medskip

\ni 
\end{rmk}

\begin{enumerate}
\item $H_1(F;\Z)$ has torsion iff \ $\frac{[\bG,\bG]}{[\bK,\bK]}$ \ has torsion.
\item $$b_1(F)\ = \ n\ + rk\left(\frac{[\bG,\bG]}{[\bK,\bK]}\right)$$

\end{enumerate}
\bigskip

\noindent (There are only complicated examples with torsion in higher homology of the Milnor fiber, recently found in \cite{denhamsuciu}).
\bigskip

\begin{cor}
One has
\bs

$$n \leq b_1(F) \leq n+rk(\frac{[\bG,\bG]}{[[\bG,\bG],[\bG,\bG]]})\ =\ n\ +\ rk(\frac{\bG^{(1)}}{\bG^{(2)}})$$
\end{cor}
\bs

Now we consider again the affine arrangement $\A.$ Denoting by  $\bG':=\pi_1(\MA),$ we have
$$\bG\cong \bG'\times\Z$$
where the factor $\Z$ is generated by a loop around all the hyperplanes in $\AT.$ As already said,   it follows by the Kunneth formula that if $\A$ has trivial monodromy over $\Z$ (resp. $\Q$) then $\AT$ does. 
Conversely, in  fig.\ref{completetriangle} we have an example where $\AT$ is a-monodromic but $\A$ has non-trivial monodromy.   

The a-monodromicity of $\A$ (over $\Z$) is equivalent to  
\begin{equation}\label{nomonodromy}H_1(\MA;R)\cong \left(\frac{R}{(t-1)}\right)^{n-1}\end{equation}
($R=\Z[q^{\pm 1}]$).
By considering a sequence as in (\ref{milnor0})
\begin{equation}\label{milnor0p}1\stackrel{}{\too} \bK' \stackrel{}{\too} \bG' \stackrel{\psi}{\too} \Z \stackrel{}{\too} 1\end{equation}
we can repeat the above arguments: in particular condition (\ref{nomonodromy}) is equivalent to 
$$H_1(\bK';\Z)=\frac{\bK'}{[\bK',\bK']}= \Z^{n-1}$$
and we get an exact sequence like in (\ref{milnor1}) for $\bK'$ and $\bG'.$ 
So we obtain 
\bigskip

\begin{teo}\label{amonoaff} The arrangement $\A$ is a-monodromic over $\Z$ (resp. over $\Q$) iff 
$$\frac{[\bG',\bG']}{[\bK',\bK']}\ =\ 0 \mbox{ (resp. $\frac{[\bG',\bG']}{[\bK',\bK']}\otimes\Q =\ 0$).} $$
\end{teo}
\bs

By considering a presentation for $\bG'$ 
$$1\stackrel{}{\too} \bN' \stackrel{}{\too} \bF' \stackrel{\pi}{\too} \bG' \stackrel{}{\too} 1$$
where $\bF'$ is the group freely generated by $\beta_1,\dots,\beta_n,$ we have a diagram similar to (\ref{diagram2}) for $\bG'.$  From
$$\bN'\subset [\bF',\bF'] \subset \kf $$
we have isomorphisms

$$\frac{[\bG',\bG']}{[\bK',\bK']}\ \cong\ \frac{\pi^{-1}[\bG',\bG']}{\pi^{-1}[\bK',\bK']}\ \cong\ 
\frac{[\bF',\bF']}{\bN'[\kf,\kf]}$$
which gives again theorem \ref{teo:suZ}.
\bigskip

Corollary \ref{cor:central} extends clearly to the affine case: therefore, if one line of $\A$ is in general position with respect to the others, then $\A$ is a-monodromic. 

This result has the following  useful generalization, which has both a central and an affine versions. We give here the affine one.

\begin{teo} \label{teo:direct}  Assume that the fundamental group $G'$ decomposes as a direct product
$$G'\ =\ A\times B$$
of two subgroups, each one having at least one element of length one. Then $\A$ is a-monodromic.\\
In particular, this applies to the case when $G'$ decomposes as a direct product of free groups,
$$G'\ =\ F_{i_1}\times F_{i_2}\times \dots \times F_{i_k}$$
where (at least) two of them have an element of length one. 
\end{teo}

\proof  First, remark that any commutator  $[ab,a'b']\in[G',G']$  equals $[a,a'][b,b'].$ Therefore it is sufficient  to show that  $[A,A]\subset [K',K'],$ and $[B,B]\subset[K',K'].$  

Let $a_0\in A,\ b_0\in B$ be elements of length one.  
Let $l=\psi(a), \ l'=\psi(a')$ be the lengths of $a$ and $a'$ respectively. Then
$$[a,a']\ =\ [a b_0^{-l},a'b_0^{-l'}]$$
and the second commutator lies in $[K',K']$ by construction. This proves that $[A,A]\subset [K',K'].$

In the same way, by using $a_0,$ we show that $[B,B]\subset [K',K'].$    \qed
\bigskip

\begin{rmk} This theorem includes the case when the arrangement is a disjoint union 
$\A=\A'\sqcup\A''$ of two subarrangements which intersect each other transversally. It is known that $\pi_1(\MA)$ is the direct product of $\pi_1(\M(\A'))$ with $\pi_1(\M(\A''))$ (see \cite{oka}) therefore by theorem \ref{teo:direct}  the arrangement $\A$ is a-monodromic.
This remark also seems new in the literature. 
\end{rmk}

We can use this result (or even corollary \ref{cor:central}) to prove the a-monodromicity of those examples in part 6 for which the fundamental group splits as a direcy product of free groups. 

Another example is given by any affine arrangement having only double points: in this case \ $\A=\cup_{i=1}^k \ \A_i$ \  where the $\A_i$'s are sets of parallel lines. Then $\pi_1(\A)=\times_{i=1}^k\ F_{n_i}$ where  $F_{n_i}$ is the free group in 
$n_i=|A_i |$ generators.  This gives an easy prove of the following known fact:  if there exists a line in a projective arrangement $\A$ which contains all the points of multiplicity $\geq 3,$ then $\A$ is a-monodromic.

To take care also of examples as that in fig.\ref{fig2}, where the fundamental group is not a direct product of free groups,   
let us introduce another class of graphs $\tilde{\G}$ as follows. Let the affine arrangement $\A$ have $n$ lines. Then: 
\begin{enumerate}
\item the vertex set of $\tilde{\G}$ corresponds to the set of generators $\{ \beta_i,\ i=1,\dots,n\}$ of $G';$  
\item for each edge $(\beta_i,\beta_j)$ of $\tilde{\G},$  the commutator  $[\beta_i,\beta_j]$ belongs to $[K',K'];$
\item $\tilde{\G}$   is connected.
\end{enumerate}
\medskip

\ni We call a graph $\tilde{\G}$ satisfying the previous conditions an \emph{admissible} graph. 

\bigskip

\begin{teo} \label{teo:admissible} If $\A$ allows an admissible graph  $\tilde{\G}$ then $\A$ is a-monodromic. \end{teo}
\bigskip

We need  the following lemma.
\bigskip

\begin{lem}\label{lem:ciclicprod} Let $F_n=F[\beta_1,\dots,\beta_n]$ be the free group in the generators $\beta_i$'s. Let $\varphi$ be the length function (see part 7) on $F_n.$  Then for any sequence of indices $i_0,\dots,i_k$ one has 
$$[\beta_{i_0},\beta_{i_1}][\beta_{i_1},\beta_{i_2}]\dots[\beta_{i_{k-1}},\beta_{i_k}][\beta_{i_k},\beta_{i_0}] \in [\ker(\varphi),\ker(\varphi)]$$
for each "closed" product of commutators. 
\end{lem}

\ni{\it Proof of lemma.} If $k\leq 2$ the result is trivial. If $k=3,$ a straighforward application of  Blanchfield theorem ([4]) gives the result. For $k>3,$ we can write
$$[\beta_{i_0},\beta_{i_1}][\beta_{i_1},\beta_{i_2}]\dots[\beta_{i_{k-1}},\beta_{i_k}][\beta_{i_k},\beta_{i_0}]=
([\beta_{i_0},\beta_{i_1}][\beta_{i_1},\beta_{i_2}][\beta_{i_2},\beta_{i_0}])([\beta_{i_0},\beta_{i_2}]\dots[\beta_{i_{k-1}},\beta_{i_k}][\beta_{i_k},\beta_{i_0}])$$
and we conclude by induction on $k.$  \qed
\bigskip
\begin{rmk}\label{rmk:ciclicprod} Clearly,  lemma \ref{lem:ciclicprod} applied to the generators of $G'$ gives that
$$[\beta_{i_0},\beta_{i_1}][\beta_{i_1},\beta_{i_2}]\dots[\beta_{i_{k-1}},\beta_{i_k}][\beta_{i_k},\beta_{i_0}] \in [K',K']$$ 
for each closed product of commutators. 
\end{rmk}
\bigskip
 
\ni{\it Proof of theorem \ref{teo:admissible}.} \ According to theorem \ref{amonoaff}  what we have to prove is that any commutator $[\beta_i,\beta_j]$ belongs to $[K',K'].$  

If $i, j$ corresponds to an edge $(\beta_i,\beta_j)$ of $\tilde{\G},$ the result follows by definition. 
Otherwise, let $\beta_i=\beta_{i_0},\beta_{i_1},\dots,\beta_{i_k}=\beta_j$ be a path in $\tilde{\G}$ connecting $\beta_i$ with $\beta_j.$  By definition, $[\beta_{i_j},\beta_{i_{j+1}}]\in [K',K'],$ $j=0,\dots,k-1,$ so
$\prod_{j=0}^{k-1} [\beta_{i_j},\beta_{i_{j+1}}]  \ \in [K',K'].$   By lemma \ref{lem:ciclicprod} and remark \ref{rmk:ciclicprod} 
$$[\beta_{i_0},\beta_{i_1}][\beta_{i_1},\beta_{i_2}]\dots[\beta_{i_{k-1}},\beta_{i_k}][\beta_{i_k},\beta_{i_0}] \in [K',K'].$$ 
It follows that $[\beta_{i_0},\beta_{i_k}]=[\beta_{i},\beta_{j}]\in [K',K'],$ which gives the thesis. \qed
\bigskip

We can use theorem \ref{teo:admissible} to prove conjecture (1) under further hypotheses.

\begin{cor}\label{cor:admissible} Let $\A$ be an affine arrangement and let $\G$ be its associated graph of double points.  Assume that  $\G$ contains an admissible spanning tree $\tilde{\Gamma}.$ Then $\A$ is a-monodromic. \end{cor}  \qed

\ni Of course, under the hypotheses of corollary \ref{cor:admissible}, the graph $\G$ is connected.

Examples where $\G$ contains an admissible spanning tree are the conjugate-free arrangements in definition \ref{conjfree}. Here all commutators (corresponding to the edges of $T$) of the geometric generators  are simply equal to $1$ in the group $G'.$  Therefore theorem \ref{teo:admissible} is a generalization of theorem \ref{teo:conjfree}.

Very little effort is needed to show that the whole graph $\G$ of double points in the arrangement of fig.\ref{fig2} is admissible: therefore corollary \ref{cor:admissible} applies to this case.

For the sake of completeness, we also mention that, for all the examples in part 6 which have non trivial monodromy,  all the quotient groups $[G',G']/[K',K']$ are free abelian of rank $2.$ This fact is in accordance  with the monodromy computations given in part 6, since in all these cases one has $\varphi_3$-torsion. It also follows that, for such examples,  the first homology group of the Milnor fiber has no torsion.

\begin{rmk} When the graph $\G$ of double points is not connected, then we can consider its decomposition into connected components $\G= \sqcup_i\ \G_i.$ We have a corresponding decomposition $\A=\sqcup_i\ \A_i$ of the arrangement.  By definition, every double point of $\A$ is a double point  of exactly one of the $\A_i$'s, while each  pair of lines in different  $\A_i$'s  either intersect in some point of multiplicity greater than $2,$  or are parallel (we are considering the affine case here).  If our conjecture is true, then each $\A_i$ is a-monodromic. At the moment we are not able to speculate about how the monodromy of $\A$ is influenced by these data: apparently, the only knowledge of such decomposition gives little control on the multiplicities of the intersection points of different components, which can assume very different values.  We are going to address these interesting  problems in future work.
\end{rmk}

\section*{Acknowledgments} 
Partially supported by INdAM and by: Universit\`a di Pisa under the ``PRA  - Progetti di Ricerca di Ateneo'' (Institutional Research Grants) - Project no. PRA\_2016\_67 
``Geometria, Algebra e Combinatoria di Spazi di Moduli e Configurazioni''.

\vfill\eject
%
%

%
%
%

\end{document}